\documentclass[a4wide,12pt,oneside]{article}

\usepackage{amsmath}
\usepackage{amssymb}
\usepackage{amsthm}

\usepackage{t1enc}
\usepackage[english]{babel}
\usepackage{verbatim} 
\usepackage{graphicx,bm,xcolor}
\usepackage[dvips]{epsfig}
\usepackage{url}
\usepackage{tikz}
\usetikzlibrary{shapes}
\usepackage{bigdelim}
\usepackage{subcaption}
\usepackage{booktabs, tabularx}

\theoremstyle{plain} 
\newtheorem{theorem}{Theorem}[section]

\newtheorem{lemma}[theorem]{Lemma}

\newtheorem{remark}[theorem]{Remark}

\theoremstyle{definition} %
\newtheorem{assumption}{Assumption}

\theoremstyle{remark} %

\title{A IETI-DP method for discontinuous Galerkin discretizations in Isogeometric Analysis with inexact local solvers}
\author{{Monica Montardini\footnote{Università degli Studi di Pavia, \texttt{monica.montardini@unipv.it}},
Giancarlo Sangalli\footnote{Università degli Studi di Pavia, \texttt{giancarlo.sangalli@unipv.it}},
Rainer Schneckenleitner\footnote{Tampere University, \texttt{rainer.schneckenleitner@tuni.fi}},
}\\{
Stefan Takacs\footnote{Johannes Kepler University Linz, \texttt{stefan.takacs@numa.uni-linz.ac.at}},
and
Mattia Tani\footnote{IMATI-CNR, \texttt{mattia.tani@imati.cnr.it}}
}}
\date{\today}

\begin{document}
\maketitle

%
%
We construct solvers for an isogeometric multi-patch discretization, where
the patches are coupled via a discontinuous Galerkin approach, which allows
the consideration of discretizations that do not
match on the interfaces. We solve the resulting linear system using
a Dual-Primal IsogEometric Tearing and Interconnecting (IETI-DP) method.
We are interested in solving the arising patch-local problems using
iterative solvers since this allows the reduction of the memory footprint.
We solve the patch-local problems approximately using the Fast
Diagonalization method, which is known to be robust in the grid size and
the spline degree. To obtain the tensor structure
needed for the application of the Fast Diagonalization method, we
introduce an orthogonal splitting of the local function spaces.
We present a convergence theory that confirms that the
condition number of the preconditioned system only grows poly-logarithmically
with the grid size. The numerical experiments confirm this finding. Moreover, they
show that the convergence of the overall solver only mildly
depends on the spline degree. We observe a mild reduction of the computational times and a significant reduction
of the memory requirements in comparison to standard IETI-DP solvers using
sparse direct solvers for the local subproblems. Furthermore, the experiments
indicate good scaling behavior on distributed memory machines.

%
%
\section{Introduction}
\label{sec:1}
Isogeometric Analysis (IgA), introduced in the seminal paper~\cite{CottrellHughes:2005}, is an approach to solve Partial Differential Equations (PDEs). IgA has been invented to enhance the interaction of Computer Aided Design (CAD) and numerical simulation. While in standard Finite Element Methods (FEM), the computational domain is usually meshed, IgA is directly based on the parametrization of the computational domain in terms of a geometry function.
In practice, the geometry is usually composed of several patches (multi-patch IgA) and each patch has its own geometry function. To be able to set up continuous Galerkin methods, both
the geometry function and the function spaces have to agree on the interface. This is not
the case for discontinuous Galerkin methods,
which allows to consider non-matching geometries and function spaces.
We focus on the Symmetric Interior Penalty discontinuous Galerkin (SIPG) approach, cf.~\cite{Arnold:1982}. It is worth noting that the SIPG method even allows geometry representations, where inaccuracies in the representation yield gaps between or overlaps of neighboring patches, cf.~\cite{HoferLanger:2019b}. The interfaces between the patches are not necessarily conforming, i.e., T-junctions are possible, cf.~\cite{SchneckenleitnerTakacs:2022}.

The discretization yields a large-scale linear system that should be solved efficiently.
Since we consider non-overlapping geometries, iterative substructuring solvers are an obvious choice. 
The Finite Element Tearing and Interconnecting (FETI) method, introduced in~\cite{FarhatRoux:1991a}, 
is one of the most powerful substructuring solvers. Particularly powerful are the
Dual-Primal FETI (FETI-DP) methods. The idea of FETI has already been adapted to IgA,
cf.~\cite{Kleiss:2012,Hofer:2019a,SchneckenleitnerTakacs:2019,BozyMontardiniSangalliTani:2020},
and is sometimes called Dual-Primal IsogEometric Tearing and Interconnecting (IETI-DP) method.

We consider a IETI-DP method for SIPG discretizations of PDEs, cf.~\cite{HoferLanger:2017c}. In the recent papers~\cite{SchneckenleitnerTakacs:2020,SchneckenleitnerTakacs:2022}, we have developed a convergence analysis that gives
condition number estimates of the preconditioned IETI system in terms of the mesh size and the spline degree. In that paper, we have only considered the case that all local subproblems are solved
with sparse direct solvers.

According to~\cite{BadiaMartin:2015}, there is a trend to use more cores with less memory per core. This trend demands for less memory consuming solvers to avoid memory overflows. Regarding memory cost, iterative solvers are more advantageous in comparison to direct solvers. FETI-DP type methods with iterative solvers for continuous Galerkin discretizations have been introduced, e.g. in~\cite{KlawonnRheinbach:2007}. There, the authors showed results for the local Neumann problems as well as an approach to use an iterative method for the coarse grid problem. For BDDC, which is closely related to FETI-DP type methods, inexact variants are introduced e.g. in~\cite{LiWidlund:2007,Dohrmann:2007} and the references therein.

We are interested in applying such a domain decomposition solver to SIPG discretizations in IgA.
We present a preconditioning strategy to approximately solve the local Neumann problems
using the Fast Diagonalization (FD) method, which satisfies condition number bounds
that are robust in the spline degree and in the grid size, cf.~\cite{SangalliTani:2016}. The FD method requires that the
discretization has tensor-product structure. For the local Neumann problems, this structure is
broken due to the elimination of the primal degrees of freedom and due to the jump terms
in the SIPG formulation. We introduce an orthogonal splitting of the degrees of freedom into a
space with tensor-product structure and a small space of remaining functions.
The first space allows the use of the FD method. For the second space, we can afford a direct
solver. We show that -- concerning the dependence on the grid size -- the resulting IETI-DP method 
satisfies the same convergence number bounds as the IETI-DP variant with direct local solvers
analyzed in~\cite{SchneckenleitnerTakacs:2020}. Concerning the dependence on the spline degree,
the numerical results show that the dependence on the spline degree is comparable to that of
the variant analyzed in~\cite{SchneckenleitnerTakacs:2020}.

The remainder of the paper is organized as follows: In Section~\ref{sec:2}, we give a description of the model problem and define the considered discretization. In the subsequent Section~\ref{sec:3}, we introduce the IETI-DP solver together with the proposed preconditioner. In the last part of this section we state an outline of the solution algorithm. Section~\ref{sec:4} is devoted to the convergence analysis. Numerical experiments are provided in Section~\ref{sec:5}. In the last Section~\ref{sec:6}, we draw some conclusions. 

\section{Preliminaries}
\label{sec:2}
\subsection{The model problem and its discretization}
We consider the following model problem: given a bounded Lipschiz domain $\Omega \subset \mathbb{R}^2$ and a source function $f \in L^2(\Omega)$, we want to find a function $u \in H^1_0(\Omega)$ such that 
\begin{align}
	\label{prob:continuous}
	\int_{\Omega} \nabla u \cdot \nabla v \; \mathrm{d}x = \int_{\Omega} fv \; \mathrm{d}x \quad \mbox{ for all } v \in H^1_0(\Omega).
\end{align}
Here and in what follows, the spaces $L^2(\Omega)$ and $H^1(\Omega)$ denote the usual Lebesgue and Sobolev spaces, respectively. These spaces are equipped with the usual norms $\| \cdot \|_{L^2(\Omega)}$, $\| \cdot \|_{H^1(\Omega)}$ and seminorm $| \cdot |_{H^1(\Omega)}$. The space $H^1_0(\Omega) \subset H^1(\Omega)$ is the subspace of functions which vanish on the (Dirichlet) boundary $\partial \Omega$. We assume that the computational domain is the union of $K$ non-overlapping patches $\Omega^{(k)}$, i.e.,
\[
\overline{\Omega} = \bigcup_{k = 1}^{K} \overline{\Omega^{(k)}}, \qquad
\Omega^{(k)} \cap \Omega^{(\ell)} = \emptyset  \mbox{ for all } k \neq \ell, 
\] 
where $\overline{T}$ denotes the closure of $T$.
Each patch $\Omega^{(k)}$ is parameterized by a geometry function 
\[
G_k: \widehat{\Omega} := (0,1)^2 \mapsto \Omega^{(k)} := G_k(\widehat{\Omega}) \subset \mathbb{R}^2,
\]
which can be continuously extended to $\overline{\widehat{\Omega}}$. {Throughout the paper, we denote any entity associated to the parameter domain with a hat symbol.}
As in \cite{SchneckenleitnerTakacs:2020}, we assume that there are
no T-junctions and that the geometry functions do not have singularities. This is formalized by the following two assumptions.
\begin{assumption}
	\label{ass:common element}
	For any $k \neq \ell$, the intersection $\overline{\Omega^{(k)}} \cap \overline{\Omega^{(\ell)}}$ is either a common vertex, a common edge (including the vertices) or the empty set.
\end{assumption}
\begin{assumption}
	\label{ass:non-singularity}
	There is a constant $C_1 > 0$ such that the patch diameters $H_{k} > 0$ satisfy 
	\[
	\| \nabla G_k\|_{L^\infty(T)} \leq C_1 H_{k} \quad \text{and} \quad \|( \nabla G_k)^{-1}\|_{L^\infty(T)} \leq C_1 \frac{1}{H_{k}}
	\]
	for $T \in \{\widehat{\Omega}, \partial \widehat{\Omega}\}$ and all $k = 1, \dots, K$.
\end{assumption}

The edge shared by two patches $\Omega^{(k)}$ and $\Omega^{(\ell)}$ is denoted by $\Gamma^{(k,\ell)}$ and its pre-image by $\widehat{\Gamma}^{(k,\ell)} := G_k^{-1}(\Gamma^{(k,\ell)})$. If a patch $\Omega^{(k)}$ contributes to the (Dirichlet) boundary, we denote the corresponding edges of the parameter domain by $\widehat{\Gamma}_D^{(k)} = G_k^{-1}(\partial \Omega^{(k)} \cap \partial \Omega)$. 

For each patch $\Omega^{(k)}$, the indices of the neighboring patches $\Omega^{(\ell)}$ are collected in the set 
\[
\mathcal{N}_\Gamma(k) := \{\ell \neq k:\Omega^{(k)} \text{ and } \Omega^{(\ell)} \text{ share a common edge} \}.
\]

In the following, we describe the involved function spaces. Let $p \in \mathbb{N} := \{ 1,2,3, \dots \}$ be a given spline degree. For better readability, we use the same spline degree throughout the whole domain.
Let $\Xi := \left(\xi_1, \dots, \xi_{n+p+1} \right)$ be a $p$-open knot vector with knots $0=\xi_1 \leq \xi_2 \leq \dots \leq \xi_{n+p+1} = 1$.
The Cox-de Boor formula, see, e.g., \cite{Boor}, provides a B-spline basis $(B[p, \Xi, i])_{i = 1}^n$ for the univariate B-spline space
\[
\mathcal{S}[p,\Xi] := \text{span} \{B[p, \Xi, 1], \dots, B[p, \Xi, n]\}
\]
with mesh size $\widehat{h} = \max_{i = 1, \dots, n+p} \{|\xi_{i+1} - \xi_{i}|\}$.
For any $k \in \{ 1, \dots, K \}$, we define the tensor product B-spline space $\widehat{V}^{(k)}$ on the parameter domain with the $p$-open knot vectors $\Xi_1^{(k)}$ and $\Xi_2^{(k)}$, as
\[
\widehat{V}^{(k)} := \{ \widehat{v}^{(k)} \in \mathcal{S}[p, \Xi_1^{(k)}] \otimes \mathcal{S}[p, \Xi_2^{(k)}]: \widehat{v}^{(k)}|_{\widehat{\Gamma}_D^{(k)}} = 0 \},
\]
where $v|_T$ denotes the restriction of $v$ to $T$
and $\otimes$ denotes the tensor product. 
A basis for $\widehat{V}^{(k)}$ is given by the tensor-product B-spline basis functions that vanish on the Dirichlet boundary, which we denote by
$
\widehat{\Phi}^{(k)} := (\widehat{\phi}^{(k)}_i)_{i = 1}^{N^{(k)}}
$.
For a formal introduction of the basis $\widehat{\Phi}^{(k)}$, we refer  to~\cite{SchneckenleitnerTakacs:2020}.

We assume that the grids are quasi-uniform.
\begin{assumption}
	\label{ass:quasi-uniformity}
	There is a constant $C_2 > 0$ such that for every $k = 1, \dots, K$, 
	\[
	C_2 \widehat{h}_{k} \leq \xi_{\delta, i+1}^{(k)} - \xi_{\delta, i}^{(k)} \leq \widehat{h}_{k}
	\]
	holds on every non-empty knot span $(\xi_{\delta, i}^{(k)}, \xi_{\delta, i+1}^{(k)})$ with $i = 1, \dots, n_{\delta}^{(k)}+p$ for the parametric directions $\delta = 1,2$.
\end{assumption}

The discrete function spaces  on the physical domain and their bases
are defined using the pull-back principle, i.e.,
\begin{align}
	\nonumber
	V^{(k)} &:= \{v^{(k)} = \widehat{v}^{(k)} \circ {G}_k^{-1} : \widehat{v}^{(k)} \in \widehat{V}^{(k)} \},
	\\
	\label{def:ordered physical basis}
	\Phi^{(k)} &:= (\phi_i^{(k)})_{i=1}^{N^{(k)}} \quad \text{ with } \quad \phi_i^{(k)} := \widehat{\phi}_i^{(k)} \circ G_k^{-1}.
\end{align}
We call $h_{k} := \widehat{h}_{k}H_{k}$ the grid size on the physical
domain. 

We consider the
Symmetric Interior Penalty discontinuous Galerkin (SIPG) formulation of the problem \eqref{prob:continuous},
which reads as follows. Find $u := (u^{(1)},\ldots,u^{(K)}) \in V := V^{(1)} \times \dots \times V^{(K)}$ such that
\begin{align}
	\label{discreteVarProb}
	a_h(u,v)= \langle f,v\rangle  
	\quad \mbox{for all}\quad v \in V,
\end{align}
where 
\begin{align*}
	a_h(u,v) &:= \sum_{k=1}^{K} \left( 
	a^{(k)}(u,v) + m^{(k)}(u,v) + r^{(k)}(u,v)
	\right),\\
	a^{(k)}(u,v) &:= \int_{\Omega^{(k)}} \nabla u^{(k)} \cdot \nabla v^{(k)} \; \textrm{d}x, \\
	m^{(k)}(u,v) &:= \sum_{\ell \in \mathcal{N}_\Gamma(k)} \int_{\Gamma^{(k,\ell)}} \frac{1}{2} 
	\left( 
	\frac{\partial u^{(k)}}{\partial n_k}(v^{(\ell)} - v^{(k)}) + 
	\frac{\partial v^{(k)}}{\partial n_k}(u^{(\ell)} - u^{(k)})
	\right) \; \textrm{d}s, \\
	r^{(k)}(u,v) &:= \sum_{\ell \in \mathcal{N}_\Gamma(k)} \int_{\Gamma^{(k,\ell)}} \frac{\delta}{h_{k\ell}}   
	(u^{(\ell)} - u^{(k)})(v^{(\ell)} - v^{(k)}) \; \textrm{d}s,\\
	\langle f,v\rangle &:= \sum_{k=1}^K \int_{\Omega^{(k)}} fv^{(k)} \; \mathrm{d}x ,
\end{align*}
$h_{k\ell}:=\min\{h_k, h_\ell\}$, $n_k$ is the outward unit normal vector on $\partial \Omega^{(k)}$, and $\delta > 0$ is some suitably chosen penalty parameter. We choose $\delta$ such that the problem~\eqref{discreteVarProb} is bounded and coercive in the dG-norm $\| \cdot \|_{d}$, induced by the scalar product
\[
d(u,v) := \sum_{k=1}^{K} d^{(k)}(u,v) = \sum_{k=1}^{K} \left( a^{(k)}(u,v) + r^{(k)}(u,v) \right).
\] 
\cite[Theorem~3.3]{Takacs:2019} guarantees that there is a choice of $\delta > 0$ independent of $h_k$ and $H_k$ such that the bilinear form is coercive and bounded. 

\section{The IETI-DP solver}
\label{sec:3}
\subsection{The overall solver}

This section is devoted to the introduction of the IETI-DP solver for the problem~\eqref{discreteVarProb}. We start with the introduction of the local function spaces that are required for the setup of
the IETI-DP solver. For discontinuous Galerkin discretizations, an appropriate choice is not completely straight forward. We follow the approach that has been proposed in \cite{DryjaGalvis:2013}
for standard FEM and extended to IgA in \cite{HoferLanger:2017c,SchneckenleitnerTakacs:2020}. 

The local function spaces are given by 
\begin{align}
	\label{def:extended space}
	V_e^{(k)} := V^{(k)} \times \prod_{\ell \in \mathcal{N}_\Gamma(k)} V^{(k,\ell)}, \qquad \mbox{for } k=1,\dots,K,
\end{align} 
where $V^{(k,\ell)}$ is the trace of $V^{(\ell)}$ on $\Gamma^{(k,\ell)}$.
A function $v_e^{(k)} \in V_e^{(k)}$ is composed of entities 
\begin{align}
	\label{def:representation}
	v_e^{(k)} = \left( 
	v^{(k)}, (v^{(k,\ell)})_{\ell \in \mathcal{N}_\Gamma(k)}
	\right),
\end{align}
where $v^{(k)} \in V^{(k)}$ and $v^{(k,\ell)} \in V^{(k,\ell)}$.
Here and in what follows, we use this notation to refer to
components of any function, like $u_e^{(k)}$, $\widehat u_e^{(k)}$
or $\widehat v_e^{(k)}$.

Let $\Phi^{(k,\ell)}$ be the collection of traces of basis functions in $\Phi^{(\ell)}$ that do not vanish on $\Gamma^{(k,\ell)}$. 
Note that $\Phi^{(k,\ell)}$ forms a basis of $V^{(k,\ell)}$ and 
we define a basis $\Phi_e^{(k)}=(\phi_{e,i}^{(k)})_{i=1}^{N_e^{(k)}}$
for the space $V_e^{(k)}$ by collecting
\begin{itemize}
	\item the basis functions in the basis $\Phi^{(k)}$ and
	\item for each $\ell\in\mathcal N_\Gamma(k)$, the basis functions
	in the bases $\Phi^{(k,\ell)}$.
\end{itemize}
This choice is visualized in Figure~\ref{fig:ai}, where each symbol represents one basis function. A formal introduction of the basis
is given in~\cite{SchneckenleitnerTakacs:2020}.

\begin{figure}[th]
	\begin{center}
		\resizebox{20em}{!}{
			\begin{tikzpicture}
				\def\shift{0.5}
				\fill[gray!20] (-0.2,0) -- (1.5,0) -- (1.5,1.7) -- (-0.2,1.7);
				\fill[gray!20] (-0.2,-1.5-\shift) -- (1.5,-1.5-\shift) -- (1.5,-3.2-\shift) -- (-0.2,-3.2-\shift);
				\fill[gray!20] (4.7+\shift,0) -- (3.0+\shift,0) -- (3.0+\shift,1.7) -- (4.7+\shift,1.7);
				\fill[gray!20] (4.7+\shift,-1.5-\shift) -- (3.0+\shift,-1.5-\shift) -- (3.0+\shift,-3.2-\shift) -- (4.7+\shift,-3.2-\shift);
				
				\draw (-0.2,0) -- (1.5,0) -- (1.5,1.7) node at (.25,2) {$\Omega^{(1)}$};
				\draw (-0.2,-1.5-\shift) -- (1.5,-1.5-\shift) -- (1.5,-3.2-\shift) node at (.25,-3.5-\shift) {$\Omega^{(2)}$};
				\draw (4.7+\shift,0) -- (3.0+\shift,0) -- (3.0+\shift,1.7) node at (4.75+\shift,2) {$\Omega^{(3)}$};
				\draw (4.7+\shift,-1.5-\shift) -- (3.0+\shift,-1.5-\shift) -- (3.0+\shift,-3.2-\shift) node at (4.75+\shift,-3.5-\shift) {$\Omega^{(4)}$};
				
				\node at (-0.7,1.3) {$V^{(1)}$};
				\node at (-0.7, -0.4) {$V^{(1,2)}$};
				\node at (2.0, 2.0) {$V^{(1,3)}$};
				
				\node at (-0.7,-2.8-\shift) {$V^{(2)}$};
				\node at (-0.7,-1.1-\shift) {$V^{(2,1)}$};
				\node at (2.0,-3.5-\shift) {$V^{(2,4)}$};
				
				\node at (5.2+\shift,1.3) {$V^{(3)}$};
				\node at (2.6+\shift,2.0) {$V^{(3,1)}$};
				\node at (5.2+\shift,-0.4) {$V^{(3,4)}$};
				
				\node at (5.2+\shift,-2.8-\shift) {$V^{(4)}$};
				\node at (5.2+\shift,-1.1-\shift) {$V^{(4,3)}$};
				\node at (2.6+\shift,-3.5-\shift) {$V^{(4,2)}$};
				
				\draw (-0.2,-0.4) -- (1.5,-0.4) {};
				\draw (1.9,0) -- (1.9,1.7) {};
				
				\draw (3.0+\shift,-0.4) -- (4.7+\shift,-0.4) {};
				\draw (2.6+\shift,0) -- (2.6+\shift,1.7) {};
				
				\draw (-0.2,-1.1-\shift) -- (1.5,-1.1-\shift) {};
				\draw (1.9,-1.5-\shift) -- (1.9,-3.2-\shift) {};
				
				\draw (3.0+\shift,-1.1-\shift) -- (4.7+\shift,-1.1-\shift) {};
				\draw (2.6+\shift,-1.5-\shift) -- (2.6+\shift,-3.2-\shift) {};
				
				\draw (1.5,0) node[circle, fill, inner sep = 2pt] (A7) {};
				\draw (1.5,1.0) node[circle, fill, inner sep = 2pt] (A9) {};
				\draw (0.35,0.0) node[circle, fill, inner sep = 2pt] (A11) {};
				\draw (0.35,1.0) node[circle, fill, inner sep = 2pt] (A11) {};
				
				\draw (1.9,0) node[rectangle, fill, inner xsep = 2.5pt, inner ysep = 2.5pt] (A1) {};
				\draw (1.9,0.5) node[rectangle, fill, inner xsep = 2.5pt, inner ysep = 2.5pt] (A2) {};
				\draw (1.9,1.5) node[rectangle, fill, inner xsep = 2.5pt, inner ysep = 2.5pt] (A3) {};
				\draw (0.75,-0.4) node[diamond, fill, inner sep = 2pt] (A5) {};
				\draw (1.5,-0.4) node[diamond, fill, inner sep = 2pt] (A6) {};
				
				\draw (3.0+\shift,0) node[rectangle, fill, inner xsep = 2.5pt, inner ysep = 2.5pt] (B7) {};
				\draw (3.0+\shift,0.5) node[rectangle, fill, inner xsep = 2.5pt, inner ysep = 2.5pt] (B8) {};
				\draw (3.0+\shift,1.5) node[rectangle, fill, inner xsep = 2.5pt, inner ysep = 2.5pt] (B9) {};
				\draw (3.75+\shift,0) node[rectangle, fill, inner xsep = 2.5pt, inner ysep = 2.5pt] (B10) {};
				\draw (3.75+\shift,0.5) node[rectangle, fill, inner xsep = 2.5pt, inner ysep = 2.5pt] (B10) {};
				\draw (3.75+\shift,1.5) node[rectangle, fill, inner xsep = 2.5pt, inner ysep = 2.5pt] (B10) {};
				\draw (4.5+\shift,0) node[rectangle, fill, inner xsep = 2.5pt, inner ysep = 2.5pt] (B11) {};
				\draw (4.5+\shift,0.5) node[rectangle, fill, inner xsep = 2.5pt, inner ysep = 2.5pt] (B11) {};
				\draw (4.5+\shift,1.5) node[rectangle, fill, inner xsep = 2.5pt, inner ysep = 2.5pt] (B11) {};
				
				\draw (2.6+\shift,0) node[circle, fill, inner sep = 2pt] (B1) {};
				\draw (2.6+\shift,1.0) node[circle, fill, inner sep = 2pt] (B3) {};
				\draw (3.0+\shift,-0.4) node[star, fill, inner sep = 2pt] (B4) {};
				\draw (3.75+\shift,-0.4) node[star, fill, inner sep = 2pt] (B5) {};
				\draw (4.5+\shift,-0.4) node[star, fill, inner sep = 2pt] (B6) {};
				
				\draw (1.5,-1.5-\shift) node[diamond, fill, inner sep = 2pt] (C7) {};
				\draw (1.5,-2.2-\shift) node[diamond, fill, inner sep = 2pt] (C8) {};
				\draw (1.5,-3.0-\shift) node[diamond, fill, inner sep = 2pt] (C9) {};
				\draw (0.75,-1.5-\shift) node[diamond, fill, inner sep = 2pt] (C11) {};
				\draw (0.75,-2.2-\shift) node[diamond, fill, inner sep = 2pt] (C11) {};
				\draw (0.75,-3.0-\shift) node[diamond, fill, inner sep = 2pt] (C11) {};
				
				\draw (1.9,-1.5-\shift) node[star, fill, inner sep = 2pt] (C1) {};
				\draw (1.9,-2.0-\shift) node[star, fill, inner sep = 2pt] (C2) {};
				\draw (1.9,-2.75-\shift) node[star, fill, inner sep = 2pt] (C3) {};
				\draw (0.35,-1.1-\shift) node[circle, fill, inner sep = 2pt] (C5) {};
				\draw (1.5,-1.1-\shift) node[circle, fill, inner sep = 2pt] (C6) {};
				
				\draw (3.0+\shift,-1.5-\shift) node[star, fill, inner sep = 2pt] (D7) {};
				\draw (3.0+\shift,-2.0-\shift) node[star, fill, inner sep = 2pt] (D8) {};
				\draw (3.0+\shift,-2.75-\shift) node[star, fill, inner sep = 2pt] (D9) {};
				\draw (3.75+\shift,-1.5-\shift) node[star, fill, inner sep = 2pt] (D10) {};
				\draw (3.75+\shift,-2.0-\shift) node[star, fill, inner sep = 2pt] (D10) {};
				\draw (3.75+\shift,-2.75-\shift) node[star, fill, inner sep = 2pt] (D10) {};
				\draw (4.5+\shift,-1.5-\shift) node[star, fill, inner sep = 2pt] (D11) {};
				\draw (4.5+\shift,-2.0-\shift) node[star, fill, inner sep = 2pt] (D11) {};
				\draw (4.5+\shift,-2.75-\shift) node[star, fill, inner sep = 2pt] (D11) {};
				
				\draw (2.6+\shift,-1.5-\shift) node[diamond, fill, inner sep = 2pt] (D1) {};
				\draw (2.6+\shift,-2.2-\shift) node[diamond, fill, inner sep = 2pt] (D2) {};
				\draw (2.6+\shift,-3.0-\shift) node[diamond, fill, inner sep = 2pt] (D3) {};
				\draw (3.0+\shift,-1.1-\shift) node[rectangle, fill, inner xsep = 2.5pt, inner ysep = 2.5pt] (D4) {};
				\draw (3.75+\shift,-1.1-\shift) node[rectangle, fill, inner xsep = 2.5pt, inner ysep = 2.5pt] (D5) {};
				\draw (4.5+\shift,-1.1-\shift) node[rectangle, fill, inner xsep = 2.5pt, inner ysep = 2.5pt] (D6) {};
		\end{tikzpicture}}
		\caption{Local spaces with artificial interfaces \label{fig:ai}}
	\end{center}
\end{figure}
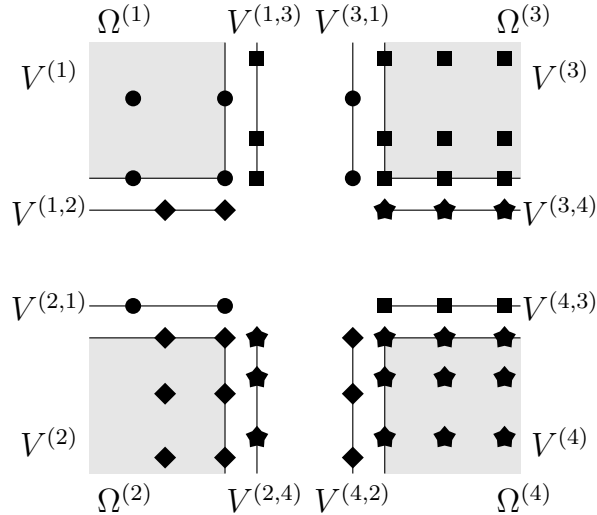

We define restrictions of the bilinear form $a_h(\cdot,\cdot)$ to the local spaces
$V_e^{(k)}$ by
\[
\begin{aligned}
	a_e^{(k)}(u_e^{(k)},v_e^{(k)}) := a^{(k)}(u_e^{(k)},v_e^{(k)}) + m^{(k)}(u_e^{(k)},v_e^{(k)}) + r^{(k)}(u_e^{(k)},v_e^{(k)}), 
\end{aligned}
\]
where, although with a slight abuse of notation,
\[
\begin{aligned}
	a^{(k)}(u_e^{(k)},v_e^{(k)}) &:= \int_{\Omega^{(k)}} \nabla u^{(k)} \cdot \nabla v^{(k)} \; \textrm{d}x, \\
	m^{(k)}(u_e^{(k)},v_e^{(k)}) &:= \sum_{\ell \in \mathcal{N}_\Gamma(k)} \int_{\Gamma^{(k,\ell)}} \frac{1}{2} 
	\left( 
	\frac{\partial u^{(k)}}{\partial n_k}(v^{(k,\ell)} - v^{(k)}) + 
	\frac{\partial v^{(k)}}{\partial n_k}(u^{(k,\ell)} - u^{(k)})
	\right) \; \textrm{d}s, \\
	r^{(k)}(u_e^{(k)},v_e^{(k)}) &:= \sum_{\ell \in \mathcal{N}_\Gamma(k)} \int_{\Gamma^{(k,\ell)}} \frac{\delta}{h_{k\ell}}   
	(u^{(k,\ell)} - u^{(k)})(v^{(k,\ell)} - v^{(k)}) \; \textrm{d}s.
\end{aligned}
\]
The localized dG-norm $\| \cdot \|_{d_e^{(k)}}$ is induced by the localized
scalar product
\[
d_e^{(k)}(u_e^{(k)}, v_e^{(k)})=a^{(k)}(u_e^{(k)}, v_e^{(k)}) + r^{(k)}(u_e^{(k)}, v_e^{(k)}).
\]
Analogously, we restrict the linear form $\langle f,\cdot\rangle$
to $V_e^{(k)}$ and define
\[
\langle f_e^{(k)},v_e^{(k)}\rangle
:=
\int_{\Omega^{(k)}} f v^{(k)} \mathrm d x.
\]

The next step is the introduction of the global discretization space
%
$
V_e := V_e^{(1)}\times\cdots\times V_e^{(K)}
$.
A function in this space is denoted by $v_e=(v_e^{(1)},\ldots,
v_e^{(K)})$. Here and in what follows, we use this notation, also
in combination with the notation~\eqref{def:representation}, to
refer to components of any function, like $u_e$, $\widehat u_e$ or
$\widehat v_e$.


Now, we set up the linear system to be solved.
Let $A^{(k)}:=[a_e^{(k)}(\phi_{e,i}^{(k)}, \phi_{e,j}^{(k)})]_{i,j=1}^{N_e^{(k)}}$ be the local
stiffness matrix and $\underline{f}_e^{(k)} := 
[\langle f_e^{(k)}, \phi_{e,i}^{(k)}\rangle]_{i=1}^{N_e^{(k)}}$ be the local load vector.
By collecting all the local contributions, we obtain the global stiffness matrix and the global load vector:
\[
A:=~ \mbox{diag}(A^{(1)}, \dots, A^{(K)}) \quad \mbox{and}\quad
\underline f:=((\underline{f}_e^{(1)})^\top, \dots, (\underline{f}_e^{(K)})^\top)^\top.
\]

Following the principle of dual-primal IETI methods, we introduce
primal constraints. We restrict
ourselves to vertex values only (for alternatives, cf.~e.g.~\cite{SchneckenleitnerTakacs:2020, Pechstein:2013a}). So,
for every corner $\mathbf{x}$ of $\Omega^{(k)}$ we enforce 
\[
v^{(k)}(\mathbf{x}) = v^{(\ell,k)}(\mathbf{x})
\]
for every $\ell \in \mathcal{N}_\Gamma(k)$, where we consider any neighboring patch $\Omega^{(\ell)}$ that shares the corner $\mathbf{x}$ with $\Omega^{(k)}$.

Moreover, we introduce a matrix $B$ that models the constraints
\begin{equation}\label{eq:B:constr}
	u^{(k)}|_{\Gamma^{(k,\ell)}} = u^{(\ell,k)}
	\quad\mbox{for}\quad k=1,\ldots,K
	\quad\mbox{and}\quad \ell\in\mathcal N_\Gamma(k)
\end{equation}
in the usual way, i.e., such that it 
consists of two non-zero entries per row, one with value $1$ and one
with value $-1$, cf. \cite{SchneckenleitnerTakacs:2020}
for a formal definition.
Since we take the vertex values as primal degrees of freedom, the
matrix $B$ does not enforce the constraint~\eqref{eq:B:constr} for the vertex-based
basis functions. The constraints are visualized in Figure~\ref{fig:ommiting}.

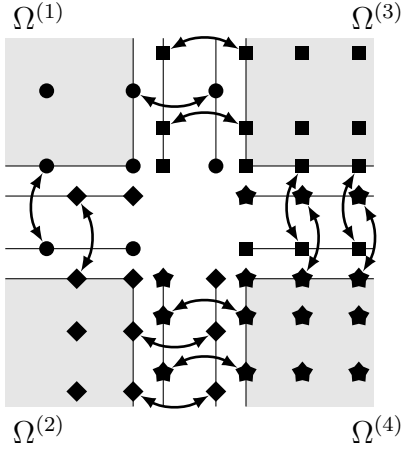
\begin{figure}
	\begin{center}
		\centering
		\begin{tikzpicture}
			\fill[gray!20] (-0.2,0) -- (1.5,0) -- (1.5,1.7) -- (-0.2,1.7);
			\fill[gray!20] (-0.2,-1.5) -- (1.5,-1.5) -- (1.5,-3.2) -- (-0.2,-3.2);
			\fill[gray!20] (4.7,0) -- (3.0,0) -- (3.0,1.7) -- (4.7,1.7);
			\fill[gray!20] (4.7,-1.5) -- (3.0,-1.5) -- (3.0,-3.2) -- (4.7,-3.2);
			
			\draw (-0.2,0) -- (1.5,0) -- (1.5,1.7) node at (0.25,2.0) {$\Omega^{(1)}$}; 
			\draw (-0.2,-1.5) -- (1.5,-1.5) -- (1.5,-3.2) node at (0.25,-3.5) {$\Omega^{(2)}$}; 
			\draw (4.7,0) -- (3.0,0) -- (3.0,1.7) node at (4.75,2.0) {$\Omega^{(3)}$}; 
			\draw (4.7,-1.5) -- (3.0,-1.5) -- (3.0,-3.2) node at (4.75,-3.5) {$\Omega^{(4)}$}; 
			
			\draw (-0.2,-0.4) -- (1.5,-0.4) {};
			\draw (1.9,0) -- (1.9,1.7) {};
			
			\draw (3.0,-0.4) -- (4.7,-0.4) {};
			\draw (2.6,0) -- (2.6,1.7) {};
			
			\draw (-0.2,-1.1) -- (1.5,-1.1) {};
			\draw (1.9,-1.5) -- (1.9,-3.2) {};
			
			\draw (3.0,-1.1) -- (4.7,-1.1) {};
			\draw (2.6,-1.5) -- (2.6,-3.2) {};
			
			\draw (1.5,0) node[circle, fill, inner sep = 2pt] (A7) {};
			\draw (1.5,1.0) node[circle, fill, inner sep = 2pt] (A9) {};
			\draw (0.35,0.0) node[circle, fill, inner sep = 2pt] (A11) {};
			\draw (0.35,1.0) node[circle, fill, inner sep = 2pt] {};
			
			\draw (1.9,0) node[rectangle, fill, inner xsep = 2.5pt, inner ysep = 2.5pt] (A1) {};
			\draw (1.9,0.5) node[rectangle, fill, inner xsep = 2.5pt, inner ysep = 2.5pt] (A2) {};
			\draw (1.9,1.5) node[rectangle, fill, inner xsep = 2.5pt, inner ysep = 2.5pt] (A3) {};
			\draw (0.75,-0.4) node[diamond, fill, inner sep = 2pt] (A5) {};
			\draw (1.5,-0.4) node[diamond, fill, inner sep = 2pt] (A6) {};
			
			\draw (3.0,0) node[rectangle, fill, inner xsep = 2.5pt, inner ysep = 2.5pt] (B7) {};
			\draw (3.0,0.5) node[rectangle, fill, inner xsep = 2.5pt, inner ysep = 2.5pt] (B8) {};
			\draw (3.0,1.5) node[rectangle, fill, inner xsep = 2.5pt, inner ysep = 2.5pt] (B9) {};
			\draw (3.75,0) node[rectangle, fill, inner xsep = 2.5pt, inner ysep = 2.5pt] (B10) {};
			\draw (4.5,0) node[rectangle, fill, inner xsep = 2.5pt, inner ysep = 2.5pt] (B11) {};
			\draw (3.75,0.5) node[rectangle, fill, inner xsep = 2.5pt, inner ysep = 2.5pt] {};
			\draw (3.75,1.5) node[rectangle, fill, inner xsep = 2.5pt, inner ysep = 2.5pt] {};
			\draw (4.5,0.5) node[rectangle, fill, inner xsep = 2.5pt, inner ysep = 2.5pt] {};
			\draw (4.5,1.5) node[rectangle, fill, inner xsep = 2.5pt, inner ysep = 2.5pt] {};
			
			\draw (2.6,0) node[circle, fill, inner sep = 2pt] (B1) {};
			\draw (2.6,1.0) node[circle, fill, inner sep = 2pt] (B3) {};
			\draw (3.0,-0.4) node[star, fill, inner sep = 2pt] (B4) {};
			\draw (3.75,-0.4) node[star, fill, inner sep = 2pt] (B5) {};
			\draw (4.5,-0.4) node[star, fill, inner sep = 2pt] (B6) {};
			
			\draw (1.5,-1.5) node[diamond, fill, inner sep = 2pt] (C7) {};
			\draw (1.5,-2.2) node[diamond, fill, inner sep = 2pt] (C8) {};
			\draw (1.5,-3.0) node[diamond, fill, inner sep = 2pt] (C9) {};
			\draw (0.75,-1.5) node[diamond, fill, inner sep = 2pt] (C11) {};
			\draw (0.75,-2.2) node[diamond, fill, inner sep = 2pt]  {};
			\draw (0.75,-3.0) node[diamond, fill, inner sep = 2pt] {};
			
			\draw (1.9,-1.5) node[star, fill, inner sep = 2pt] (C1) {};
			\draw (1.9,-2.0) node[star, fill, inner sep = 2pt] (C2) {};
			\draw (1.9,-2.75) node[star, fill, inner sep = 2pt] (C3) {};
			\draw (0.35,-1.1) node[circle, fill, inner sep = 2pt] (C5) {};
			\draw (1.5,-1.1) node[circle, fill, inner sep = 2pt] (C6) {};
			
			\draw (3.0,-1.5) node[star, fill, inner sep = 2pt] (D7) {};
			\draw (3.0,-2.0) node[star, fill, inner sep = 2pt] (D8) {};
			\draw (3.0,-2.75) node[star, fill, inner sep = 2pt] (D9) {};
			\draw (3.75,-1.5) node[star, fill, inner sep = 2pt] (D10) {};
			\draw (4.5,-1.5) node[star, fill, inner sep = 2pt] (D11) {};
			\draw (3.75,-2.0) node[star, fill, inner sep = 2pt] {};
			\draw (3.75,-2.75) node[star, fill, inner sep = 2pt] {};
			\draw (4.5,-2.0) node[star, fill, inner sep = 2pt] {};
			\draw (4.5,-2.75) node[star, fill, inner sep = 2pt] {};
			
			\draw (2.6,-1.5) node[diamond, fill, inner sep = 2pt] (D1) {};
			\draw (2.6,-2.2) node[diamond, fill, inner sep = 2pt] (D2) {};
			\draw (2.6,-3.0) node[diamond, fill, inner sep = 2pt] (D3) {};
			\draw (3.0,-1.1) node[shape=rectangle, fill, inner xsep = 2.5pt, inner ysep = 2.5pt] (D4) {};
			\draw (3.75,-1.1) node[rectangle, fill, inner xsep = 2.5pt, inner ysep = 2.5pt] (D5) {};
			\draw (4.5,-1.1) node[rectangle, fill, inner xsep = 2.5pt, inner ysep = 2.5pt] (D6) {};
			
			\draw[<->, line width = 1pt, latex-latex, bend left]
			(A2) edge (B8) (A3) edge (B9)
			(A5) edge (C11)
			(B6) edge (D11) (B5) edge (D10)
			(C2) edge (D8) (C3) edge (D9);
			
			\draw[<->, line width = 1pt, latex-latex, bend right]
			(A9) edge (B3)
			(A11) edge (C5)
			(B11) edge (D6) (B10) edge (D5)
			(C8) edge (D2) (C9) edge (D3);
		\end{tikzpicture}
		\captionof{figure}{Constraints that are modelled by the matrix B}
		\label{fig:ommiting}
	\end{center}
\end{figure}

We assume the following ordering of the basis functions in $\Phi_e^{(k)}$. First, there are the basis functions that vanish at the boundary of the patch. We denote the corresponding degrees of freedom with a subindex $\mathrm I$.
Then, there are those basis functions on the interfaces and artificial interfaces that vanish on the corners. We denote the corresponding degrees of freedom with a subindex $\mathrm B$. Finally, there are the basis functions associated to the corners. We denote the corresponding degrees of freedom with a subindex $\mathrm C$. We use the subindex
$\Gamma$ to denote the combination of $\mathrm B$ and $\mathrm C$ and the subindex $\Delta$ to denote the combination of $\mathrm I$ and $\mathrm B$.
Following this splittings, the matrices $A^{(k)}$ and $B^{(k)}$ 
are decomposed as follows:
\begin{equation}\label{eq:subdiv}
	\begin{aligned}
		A^{(k)} &= 
		\begin{pmatrix}
			A_{\mathrm{II}}^{(k)} & A_{\mathrm{IB}}^{(k)}& A_{\mathrm{IC}}^{(k)} \\
			A_{\mathrm{BI}}^{(k)} & A_{\mathrm{BB}}^{(k)}& A_{\mathrm{BC}}^{(k)} \\
			A_{\mathrm{CI}}^{(k)} & A_{\mathrm{CB}}^{(k)}& A_{\mathrm{CC}}^{(k)} \\
		\end{pmatrix}
		= 
		\begin{pmatrix}
			A_{\mathrm{II}}^{(k)} & A_{\mathrm{I}\Gamma}^{(k)} \\
			A_{\Gamma \mathrm{I}}^{(k)} & A_{\Gamma \Gamma}^{(k)}
		\end{pmatrix}
		=
		\begin{pmatrix}
			A_{\Delta\Delta}^{(k)} & A_{\Delta\mathrm{C}}^{(k)} \\
			A_{\mathrm{C}\Delta}^{(k)} & A_{\mathrm{CC}}^{(k)} \\
		\end{pmatrix}, \\
		B^{(k)} &= \begin{pmatrix} B^{(k)}_{\mathrm I} & B^{(k)}_{\mathrm B} & B^{(k)}_{\mathrm C}\end{pmatrix} = \begin{pmatrix} B^{(k)}_{\mathrm I} & B^{(k)}_\Gamma\end{pmatrix} = \begin{pmatrix} B^{(k)}_{\mathrm \Delta} & B^{(k)}_{\mathrm C}\end{pmatrix} = \begin{pmatrix} 0 & B_{\mathrm{B}}^{(k)} & 0 \end{pmatrix}.
	\end{aligned}
\end{equation}
Furthermore, let us define 
$A_{\Delta\Delta} := \mbox{diag}(A_{\Delta\Delta}^{(1)},\cdots, A_{\Delta\Delta}^{(K)})$ and
$B_\Delta := (B_\Delta^{(1)}, \cdots,B_\Delta^{(K)})$.

The next step is the introduction of an $A$-orthogonal and nodal basis representation matrix $\Psi$ in order to obtain the primal problem. We define for each patch $\Omega^{(k)}$ an $A^{(k)}$-orthogonal matrix by 
\begin{equation}\label{def:local primal basis}
	\Psi^{(k)}  := 
	\begin{pmatrix}
		\Psi_\Delta^{(k)} \\ \Psi_{\mathrm{C}}^{(k)}
	\end{pmatrix}
	= \begin{pmatrix}
		- (A_{\Delta\Delta}^{(k)})^{-1} A_{\Delta \mathrm C}^{(k)}
		\\ I
	\end{pmatrix}.
\end{equation}
Let $G^{(k)}$ denote the canonical global to local mapping matrix for the primal degrees of freedom. Then the $A$-orthogonal basis is given by 
\begin{equation}\label{def:global primal basis}
	\Psi:= 
	\begin{pmatrix}
		\Psi^{(1)} G^{(1)} \\
		\vdots \\
		\Psi^{(K)} G^{(K)}
	\end{pmatrix}.
\end{equation}



The following problem is equivalent to the
original problem \eqref{discreteVarProb}: Find
$(\underline u_\Delta^\top, \underline u_\Pi^\top, \underline \lambda^\top)^\top$ such that
\begin{align}\label{eq:star}
	\underbrace{
		\begin{pmatrix}
			A_{\Delta\Delta}  &	                     &  B_\Delta^\top           \\
			&   \Psi^\top A \Psi   & \Psi^\top B^\top         \\
			B_\Delta          &  B \Psi              & 0
		\end{pmatrix}	
	}_{\displaystyle \mathcal A :=}
	\begin{pmatrix}
		\underline{u}_\Delta  \\
		\underline{u}_\Pi  \\
		\underline{\lambda}  \\
	\end{pmatrix}
	=
	\begin{pmatrix}
		\underline{f}_\Delta  \\
		\underline{f}_\Pi  \\
		0  \\
	\end{pmatrix}
	,
\end{align}
where $\underline{f}_\Pi := \Psi^\top \underline{f}$, cf.~\cite{MandelDohrmannTezaur:2005a}.

In the following sections, we discuss how the system~\eqref{eq:star}
is to be solved. Having the solution, the overall
solution can be then computed by 
\begin{equation}\label{eq:veryfinal}
	\underline{u} = \begin{pmatrix} 
		\underline{u}_\Delta \\ 0  \end{pmatrix}  + \Psi \underline{u}_\Pi.
\end{equation}

\subsection{The local preconditioners}
\label{subsec:local preconditioners}

In this subsection, we introduce the preconditioners for $A_{\Delta\Delta}^{(k)}$, called $P^{(k)}$.
The main idea is to use an additive Schwarz splitting of the local extended IgA space on the parameter domain into two subspaces, i.e., 
\begin{align}
	\label{eq:space splitting}
	\widehat{V}_e^{(k)}:=
	\{
		({v}^{(k)} \circ G_k, ({v}^{(k,\ell)} \circ G_\ell)_{\ell \in \mathcal{N}_\Gamma(k)}) \, : \, v^{(k)} \in V^{(k)}, v^{(k,\ell)} \in V^{(k,\ell)}
	\} 
	= \widehat{V}_1^{(k)} \oplus \widehat{V}_2^{(k)},
\end{align} 
where 
\begin{align*}
	\widehat{V}_1^{(k)} := \left\{\left(\widehat{u}^{(k)}, \left(\pi^{(k,\ell)}\left(\widehat{u}^{(k)}|_{\Gamma^{(k,\ell)}}\right) \right)_{\ell \in \mathcal{N}_\Gamma(k)} \right): \widehat{u}^{(k)}\in \widehat{V}^{(k)}  \right\},
\end{align*}
$\pi^{(k,\ell)}$ is the $L^2$-orthogonal projection from $\widehat V^{(\ell,k)}$ into $ \widehat{V}^{(k, \ell)}$ and 
\begin{align*}
	\widehat{V}_2^{(k)} := \left\{\left(0, (\widehat{u}^{(k, \ell)})_{\ell \in \mathcal{N}_\Gamma(k)}\right): \widehat{u}^{(k,\ell)} \in \widehat{V}^{(k,\ell)}\right\}.
\end{align*}
In both cases,
$\widehat{V}^{(k,\ell)} := \left\{ u^{(k,\ell)} \circ G_{k}|_{\widehat{\Gamma}^{(k,\ell)}} : u^{(k,\ell)} \in V^{(k,\ell)} \right\}  $
is the trace space from the neighboring patch, mapped to the
corresponding edge on the parameter domain. 
Since it will be used in the following, we fix on $\widehat{V}^{(k,\ell)}$ the basis $\widehat{\Phi}^{(k,\ell)}$ induced by the basis $\Phi^{(k,\ell)}$ of $V^{(k,\ell)}$.
%
We choose the local bilinear forms in the spaces $\widehat{V}_1^{(k)}$ and $\widehat{V}_2^{(k)}$ as follows:
\begin{align}\label{eq:preconditioner1}
	\widetilde{d}_1^{(k)}(\widehat{u}_{1,e}^{(k)}, \widehat{v}_{1,e}^{(k)}) := \int_{\widehat{\Omega}^{}}^{} \nabla \widehat{u}_1^{(k)} \cdot \nabla \widehat{v}_1^{(k)} \; \text{d}x + 
	\alpha
	\int_{\widehat{\Omega}} \widehat{u}_1^{(k)}\widehat{v}_1^{(k)} \; \mathrm{d}x 
	\quad\mbox{for all}\quad \widehat{u}_{1,e}^{(k)},\widehat{v}_{1,e}^{(k)} \in \widehat{V}_1^{(k)}
\end{align}
and 
\begin{align}\label{eq:preconditioner2}
	\widetilde{d}_2^{(k)}(\widehat{u}_{2,e}^{(k)}, \widehat{v}_{2,e}^{(k)}) := \sum_{\ell \in \mathcal{N}_\Gamma(k)}^{}\frac{\delta}{\widehat{h}_{k\ell}}\int_{\widehat{\Gamma}^{(k,\ell)}}^{} \widehat{u}_2^{(k,\ell)} \widehat{v}_2^{(k,\ell)} \; \text{d}s, \quad\mbox{for all}\quad \widehat{u}_{2,e}^{(k)},\widehat{v}_{2,e}^{(k)} \in \widehat{V}_2^{(k)},
\end{align}
where $\alpha := 0$ if $|\partial \Omega^{(k)} \cap \partial \Omega| > 0$ and $\alpha = 
1$ otherwise.

On $\widehat{V}_1^{(k)}$, we consider the basis induced by the basis $\widehat{\Phi}^{(k)}$ of $\widehat{V}^{(k)}$, and on $\widehat{V}_2^{(k)}$ the basis induced by the basis $ \bigcup_{\ell \in \mathcal{N}_\Gamma(k)} \widehat{\Phi}^{(k,\ell)}$ of $\prod_{\ell \in \mathcal{N}_\Gamma(k)} \widehat{V}^{(k,\ell)}$.
Then, the above bilinear forms are represented by the local matrices
$\widetilde{D}_1^{(k)}$ and $\widetilde{D}_2^{(k)}$, respectively.
The matrix $\widetilde{D}_1^{(k)}$ 
is (if the degrees of freedom are ordered in a standard lexicographic ordering) the sum of Kronecker-product matrices. Specifically, it has the tensor structure of a Sylvester equation.
So, we can use the fast diagonalization (FD) method, cf. \cite{SangalliTani:2016}, to realize 
$\left(\widetilde{D}^{(k)}_{1}\right)^{-1}$. For the application of
$\left(\widetilde{D}^{(k)}_{1,\Delta\Delta}\right)^{-1}$, we first introduce the subdivision
\[
\widetilde{D}^{(k)}_{1} = 
\begin{pmatrix}
	\widetilde{D}^{(k)}_{1,\Delta\Delta} &  
	\widetilde{D}^{(k)}_{1,\Delta\mathrm C} \\ 
	\widetilde{D}^{(k)}_{1,\mathrm C\Delta} & 
	\widetilde{D}^{(k)}_{1,\mathrm C\mathrm C} 
\end{pmatrix}
=
\begin{pmatrix}
	\widetilde{D}^{(k)}_{1,\Delta\Delta} \\& 
	\widetilde{D}^{(k)}_{1,\mathrm C\mathrm C} 
\end{pmatrix}
+
\underbrace{\begin{pmatrix}
		\widetilde{D}^{(k)}_{1,\Delta\mathrm C} & 0 \\
		0 & I_{\mathrm C\mathrm C}
\end{pmatrix}}_{\displaystyle U^{(k)} := }
\underbrace{\begin{pmatrix}
		0 & I_{\mathrm C\mathrm C} \\
		\widetilde{D}^{(k)}_{1,\mathrm C\Delta} & 0
\end{pmatrix}}_{\displaystyle (V^{(k)})^\top := },
\]
where $I_{\mathrm C\mathrm C}$ is the identity matrix on the dofs with subindex $\mathrm C$.
Using the Sherman-Morrison-Woodbury formula, see~\cite{GolubLoan:2013}, we obtain
\begin{align*}
	&(\widetilde{D}^{(k)}_{1,\Delta\Delta})^{-1}
	=
	\begin{pmatrix}
		I_{\Delta\Delta} & 0
	\end{pmatrix}
	\begin{pmatrix}
		\widetilde{D}^{(k)}_{1,\Delta\Delta} \\& 
		\widetilde{D}^{(k)}_{1,\mathrm C\mathrm C} 
	\end{pmatrix}^{-1}
	\begin{pmatrix}
		I_{\Delta\Delta} \\ 0
	\end{pmatrix}
	\\&=
	\begin{pmatrix}
		I_{\Delta\Delta} & 0
	\end{pmatrix}
	\big(\widetilde{D}^{(k)}_{1}- U^{(k)}(V^{(k)})^\top\big)^{-1}
	\begin{pmatrix}
		I_{\Delta\Delta} \\ 0
	\end{pmatrix}
	\\
	&= 
	\begin{pmatrix}
		I_{\Delta\Delta} & 0
	\end{pmatrix}
	\Big(
	I + 
	(\widetilde{D}^{(k)}_1)^{-1}
	U^{(k)} \big(I - (V^{(k)})^\top (\widetilde{D}^{(k)}_1)^{-1}U^{(k)}  \big)^{-1} (V^{(k)})^\top 
	\Big)(\widetilde{D}^{(k)}_1)^{-1}
	\begin{pmatrix}
		I_{\Delta\Delta} \\ 0
	\end{pmatrix},
\end{align*}
where  $I_{\Delta\Delta}$ is the identity matrix on the dofs with subindex $\Delta$.

The application of $(\widetilde{D}_{2,\Delta \Delta}^{(k)})^{-1}$ (i.e., the matrix $(\widetilde{D}_{2}^{(k)})^{-1}$ without the corners) to a vector
is realized using a direct solver. This is feasible since these matrices
are small since they only live  on the artificial interfaces.

Then, we define
\[
P^{(k)} := 
\sum_{i = 1}^2 E_{i,\Delta \Delta}^{(k)} \left( \widetilde{D}_{i,\Delta \Delta}^{(k)} \right)^{-1} (E_{i,\Delta \Delta}^{(k)})^\top,
\quad
\mbox{where}
\quad
E_{i,\Delta \Delta}^{(k)}:= 
\begin{pmatrix}
	I_{\Delta\Delta} & 0
\end{pmatrix}
E^{(k)}_i 
\begin{pmatrix}
	I_{\Delta\Delta} \\ 0
\end{pmatrix}
\]
and $E_i^{(k)}$ is the matrix representation of the canonical embedding $\widehat{V}_i^{(k)} \rightarrow \widehat{V}_e^{(k)}$. 


\subsection{The inexact scaled Dirichlet preconditioner}

For the last block of the overall preconditioner, we set up
an inexact variant of the scaled Dirichlet preconditioner.  
We introduce the bilinear form on $\widehat{V}_e^{(k)}$
\[
\widehat{d}_e^{(k)} (\widehat{u}_e^{(k)},\widehat{v}_e^{(k)})
:= \int_{\widehat \Omega} \nabla \widehat{u}^{(k)} \cdot \nabla \widehat{v}^{(k)} \mathrm dx
+ \sum_{\ell \in \mathcal{N}_\Gamma(k)} \int_{\widehat\Gamma^{(k,\ell)}}
\frac{\delta}{\widehat h_{k\ell}}   
(\widehat u^{(k,\ell)} - \widehat u^{(k)})(\widehat v^{(k,\ell)} - \widehat v^{(k)}) \; \textrm{d}s
\]
and let $\widehat D^{(k)}$ be its matrix representation, which has the block structure 
\begin{equation}
	\begin{aligned}
		\widehat D^{(k)} = 
		\begin{pmatrix}
			\widehat{D}_{\mathrm{II}}^{(k)} & \widehat{D}_{\mathrm{I}\Gamma}^{(k)} \\
			\widehat{D}_{\Gamma \mathrm{I}}^{(k)} & \widehat{D}_{\Gamma \Gamma}^{(k)}
		\end{pmatrix}.
	\end{aligned}
\end{equation}
The inexact version of the the scaled Dirichlet preconditioner is
given by
\begin{equation}\label{eq:14a}
	\widehat{M}_{\mathrm{sD}} := B_\Gamma \mathcal{D}^{-1} \widehat{S}_D \mathcal{D}^{-1} B_\Gamma^\top,
\end{equation}  
where
\begin{equation}\label{eq:14b}
	\widehat{S}_D = \mbox{diag}(\widehat{S}^{(1)}_D, \dots, \widehat{S}_D^{(K)})
	\qquad\mbox{and}\qquad
	\widehat S_D^{(k)} = \widehat D_{\Gamma \Gamma}^{(k)} - \widehat D_{\Gamma \mathrm{I}}^{(k)} (\widehat D_{\mathrm{II}}^{(k)})^{-1} \widehat D_{\mathrm{I} \Gamma}^{(k)}.
\end{equation}  
The diagonal matrix $\mathcal{D}$
is set up based on the principle of multiplicity scaling. The entries $d_{i,i}$ of $\mathcal{D}$ are one plus the number of Lagrange multipliers acting on the corresponding degree of freedom. For more details on the scaled Dirichlet preconditioner and the setup of the Schur complements, we refer to~\cite{DryjaGalvis:2013}. 
Similarly as for $(\widetilde{D}_{1,\Delta \Delta}^{(k)})^{-1}$, the application of $(\widehat D_{\mathrm{II}}^{(k)})^{-1}$ can be realized using the FD method.

\subsection{The global preconditioner}
\label{subsec:global preconditioner}

The linear system~\eqref{eq:star} is solved with MINRES and a
block-diagonal preconditioner. The first blocks are
$P := \mbox{diag}(P^{(1)}, \dots, P^{(K)})$ as preconditioner
for $A_{\Delta\Delta}$ and a direct exact solver as
preconditioner for the primal system~$\Psi^\top A \Psi$. The last block is
the inexact version of the scaled Dirichlet preconditioner.
Concluding, the overall preconditioner~$\mathcal{P}$ is given
by
\[
\mathcal{P}:=
\left(
\begin{array}{ccc}
	P     \\
	& (\Psi^\top A \Psi)^{-1} \\
	&& \widehat M_{\mathrm{sD}}
\end{array}	
\right).
\]

\subsection{The IETI-DP algorithm}

In the following, we present the whole algorithm that allows to solve \eqref{eq:star} with 
a MINRES 
solver and the preconditioner $\mathcal{P}$.  
\begin{enumerate}
	
	\item Compute the LU factorization of $\widetilde D_{2,\Delta\Delta}^{(k)}$, as well as the eigendecompositions of the univariate Kronecker factors of $\widetilde D_1^{(k)}$ and $\widehat{D}_{\mathrm{II}}^{(k)}$, as needed for the Fast Diagonalization (FD) method (see \cite{SangalliTani:2016} for details).
	\item Compute $\Psi$ to set up the coarse grid problem. By~\eqref{def:local primal basis}, we have 
	\[
	\Psi^{(k)}= 
	\begin{pmatrix}
		-(A_{\Delta \Delta}^{(k)})^{-1} A_{\Delta \mathrm{C}}^{(k)} \\
		I 
	\end{pmatrix}.
	\]
	For each $k$, we solve the system 
	\begin{align}\label{eq:basis computation}
		A_{\Delta \Delta}^{(k)} \Psi^{(k)}_\Delta= 
		-A^{(k)}_{\Delta \mathrm{C}}
	\end{align}
	with a preconditioned conjugate gradient (PCG) method 
	using the preconditioner ${P}^{(k)}$. For the setup of $P^{(k)}$ the matrices $U^{(k)}$ and $V^{(k)}$ as well as $I_{\Delta\Delta}$ are computed in advance for the patch $k$. One application of $P^{(k)}$ requires matrix-vector products and the applications of $(\widetilde{D}_1^{(k)})^{-1}$ and of the dense but small matrix $\big(I - (V^{(k)})^\top (\widetilde{D}^{(k)}_1)^{-1}U^{(k)} \big)^{-1} (V^{(k)})^\top \Big)^{-1}$. The former application is realized using FD and the latter one is realized with a LU solver. 
	The linear system~\eqref{eq:basis computation} is solved up to a prescribed accuracy of $\varepsilon_C$ for
	the $\ell_2$-norm of the residual.
	\item Compute $\underline{f}_\Pi = \Psi^\top \underline{f}$ and the primal matrix $A_\Psi:=\Psi^\top A \Psi$.
	\item Solve the system \eqref{eq:star} with MINRES.
	Each iteration of the solver requires the execution of the following steps:
	\begin{enumerate}
		\item Compute the residual $\underline{r} = (\underline{r}_{\Delta}^\top, \underline{r}_\Pi^\top, \underline{q}^\top)^\top$, which
		only requires matrix-vector multiplications.
		\item Compute $\underline w_\Delta := P \underline r_\Delta$.
		The application of the embeddings $ E_{i,\Delta \Delta}^{(k)}$ needs matrix-vector multiplications and, in addition, the embedding from the space $\widehat{V}_1^{(k)}$ requires the solution of 1D mass problems which are solved with a sparse direct Cholesky solver. 
		\item Compute $\underline w_\Pi:=A_\Psi^{-1}\underline{r}_\Pi$
		using a sparse direct Cholesky solver.
		\item Apply the inexact scaled Dirichlet preconditioner, i.e.,
		compute $\underline{\zeta} :=
		\widehat{M}_\mathrm{sD}\underline{q}$ as defined
		in~\eqref{eq:14a} and \eqref{eq:14b}.
		This involves matrix-vector multiplications and the application
		of $(\widehat{D}_{\mathrm{II}}^{(k)})^{-1}$. The latter is
		realized using the FD method.		
	\end{enumerate} 
	The update direction for the MINRES 
	solver is 
	$
	\left( \underline{w}_\Delta^{\top}, \underline{w}_\Pi^{\top}, \underline{\zeta}^{\top} \right)^\top.
	$
	
	\item The final solution is obtained by~\eqref{eq:veryfinal}.
\end{enumerate}

Our numerical experience motivates the heuristic choice $\varepsilon_C = \frac{1}{100} \varepsilon $  as tolerance for solving systems \eqref{eq:basis computation}, where $\varepsilon$ is the prescribed tolerance for the iterative solution of \eqref{eq:star}. Indeed, we observed for larger values $\varepsilon_C$, the overall method might
not converge.

\begin{remark}
	Note that the primal basis can also be computed by solving $K$ indefinite linear systems involving the matrix $A^{(k)}$, see e.g., \cite{SchneckenleitnerTakacs:2020}. Removing the degrees of freedom corresponding to the corners of the patches allows us to compute the primal basis with PCG, because the matrix $A_\mathrm{\Delta \Delta}$ is symmetric and positive definite. 
\end{remark}

\section{Convergence analysis}
\label{sec:4}
This section is devoted to the convergence analysis for both the MINRES 
solver that is used for the solution of~\eqref{eq:star} and the PCG solver for  the system~\eqref{eq:basis computation}, see Theorem~\ref{thrm:fin}.

First, we introduce some useful notation.
If there is a constant $c>0$ that only depends on the spline degree~$p$, the choice of $\alpha$ in \eqref{eq:preconditioner1}  
and on the constants from Assumptions~\ref{ass:non-singularity} and \ref{ass:quasi-uniformity} such that $a \leq cb$, we write $a \lesssim b$. If $a \lesssim b \lesssim a$, we write $a \eqsim b$.
For two square matrices $A,B \in \mathbb{R}^{n\times n}, n \in \mathbb{N}$ we write $A \lesssim B$ if and only if 
$\underline{v}^\top A \underline{v} \lesssim \underline{v}^\top B \underline{v}$ holds for all $\underline{v}\in \mathbb{R}^n$.
We write $A \eqsim B$ if and only if $A \lesssim B \lesssim A$.

For any symmetric positive (semi-)definite matrix 
$X \in \mathbb{R}^{n \times n}, n\in \mathbb{N}$,
we define the norm $\| \cdot \|_{X}$ by
$
\| \underline{v} \|_{X}:= \sqrt{v^T X v}
$ 
for all $\underline{v} \in \mathbb{R}^{n}$.

For better readability, whenever it is clear from the context we do not write the restriction of a function to an interface explicitly, e.g., we write $\int_{\Gamma^{(k,\ell)}} u^{(k)} \; \mathrm{d}x$ instead of $\int_{\Gamma^{(k,\ell)}} u^{(k)}|_{\Gamma^{(k,\ell)}} \; \mathrm{d}x$.

Let $D^{(k)}$ be the matrix representation of the bilinear form $d_e^{(k)}(\cdot,\cdot)$. Lemma~4.2 in \cite{SchneckenleitnerTakacs:2020}, written in matrix from, states that
\begin{equation}
	\label{eq:patchwise equivalent}
	A^{(k)} \eqsim D^{(k)}.
\end{equation}

Due to \cite[Lemma 4.13]{SchneckenleitnerTakacs:2019}, we know that the matrices $D^{(k)}$ and $\widehat D^{(k)}$ are spectrally equivalent, i.e., we obtain 
\begin{equation}
	\label{eq:phys para equivalent}
	D^{(k)} \eqsim \widehat D^{(k)}.
\end{equation}

Now, we show that the spaces $\widehat{V}_1^{(k)}$ and $\widehat{V}_2^{(k)}$ are orthogonal with respect to the scalar product $\widehat{d}_e^{(k)}(\cdot, \cdot)$
for every $k=1,\dots, K$.
\begin{lemma}
	\label{lem:orthogonal splitting}
	The spaces $\widehat{V}_1^{(k)}$ and $\widehat{V}_2^{(k)}$ are  $\widehat{d}_e^{(k)}(\cdot, \cdot)$-orthogonal.
\end{lemma}
\begin{proof}
	Let $\widehat{u}_{1,e}^{(k)}=\left(\widehat{u}^{(k)}, \left(\pi^{(k,\ell)}\widehat{u}^{(k)} \right)_{\ell \in \mathcal{N}_\Gamma(k)} \right) \in \widehat{V}_1^{(k)}$ and \mbox{$\widehat{u}_{2,e}^{(k)}=\left(0, (\widehat{u}^{(k, \ell)})_{\ell \in \mathcal{N}_\Gamma(k)}\right) \in \widehat{V}_2^{(k)}$} be arbitrary but fixed. Then, 
	\[
	\begin{aligned}
		\widehat{d}_e^{(k)}(\widehat{u}_{1,e}^{(k)}, \widehat{u}_{2,e}^{(k)}) &= 
		\underbrace{\int_{\widehat{\Omega}} \nabla \widehat{u}_{}^{(k)} \cdot \nabla 0 \; \mathrm{d}x}_{\displaystyle = 0} + \sum_{\ell \in \mathcal{N}_\Gamma(k)} \frac{\delta}{\widehat{h}_{k\ell}} \int_{\widehat{\Gamma}^{(k,\ell)}}  (\pi^{(k,\ell)}\widehat{u}_{}^{(k)} - \widehat{u}_{}^{(k)}) (\widehat{u}_{}^{(k,\ell)} - 0) \; \mathrm{d}s.
	\end{aligned}
	\]
	Since $\pi^{(k,\ell)}$ is the $L^2(\widehat{\Gamma}^{(k,\ell)})$-orthogonal projection, we obtain 
	\[
	\int_{\widehat{\Gamma}^{(k,\ell)}}(\pi^{(k,\ell)}\widehat{u}_{}^{(k)} - \widehat{u}_{}^{(k)}) v^{(k,\ell)} = 0
	\qquad \mbox{for all } v^{(k,\ell)} \in \widehat{V}^{(k,\ell)}
	\]
	and thus
	$
	\widehat{d}_e^{(k)}(\widehat{u}_{1,e}^{(k)}, \widehat{u}_{2,e}^{(k)}) = 0,
	$
	which finishes the proof.
\end{proof}

Using Lemma~\ref{lem:orthogonal splitting}, we obtain 
\[
\widehat{D}^{(k)} = \sum_{i=1}^2 E_i^{(k)} \widehat{D}_i^{(k)}  (E_i^{(k)})^\top,
\]
where the matrices $\widehat{D}_i^{(k)}$, $i = 1,2$, are the matrix
representations of $\widehat{d}_e^{(k)}(\cdot, \cdot)$ for functions in $\widehat{V}_1^{(k)}$ and $\widehat{V}_2^{(k)}$, respectively, and $E_i^{(k)}$ represents the canonical embedding from \mbox{$\widehat{V}_i^{(k)}$ into $\widehat{V}_e^{(k)}$}.
The following lemma shows that the matrices $\widehat{D}_1^{(k)}$ and $\widetilde{D}_1^{(k)}$ are
equivalent up to the effect of constant functions. 

\begin{lemma}
	\label{lem:D equivalence}
	The equivalence $\widehat{D}_1^{(k)} + \widehat M^{(k)}\eqsim\widetilde{D}_1^{(k)}$ holds.
\end{lemma}
\begin{proof}
	We start by showing $\widehat{D}_1^{(k)} \lesssim \widetilde{D}_1^{(k)}$. Using the definitions, we obtain
	\[
	\begin{aligned}
		&(\underline{v}_{1,e}^{(k)})^\top \widehat{D}_1^{(k)} \underline{v}_{1,e}^{(k)} =  \widehat{d}_e^{(k)}(\widehat{v}_{1,e}^{(k)}, \widehat{v}_{1,e}^{(k)}) \\
		&\qquad= 
		\int_{\widehat{\Omega}} |\nabla \widehat{v}_1^{(k)}|^2 \; \text{d}x
		+ \sum_{\ell \in \mathcal{N}_\Gamma(k)} \frac{\delta}{\widehat{h}_{k\ell}} \int_{\widehat{\Gamma}^{(k,\ell)}} 
		(\pi^{(k,\ell)}\widehat{v}_1^{(k)}- \widehat{v}_1^{(k)})^2  \; \text{d}s
	\end{aligned}
	\] 
	for any given $\underline{v}_{1,e} \in \widehat{V}_1^{(k)}$.
	Using an approximation error estimate, cf. \cite{SandeManni:2019},
	\[
	\int_{\widehat{\Gamma}^{(k,\ell)}}^{}(\widehat{v}_1^{(k)}-(\mathcal \pi^{(k,\ell)}\widehat{v}_1^{(k)}))^2 \; \text{d}s
	\lesssim
	\widehat{h}_{k\ell}^2 \int_{\widehat{\Gamma}^{(k,\ell)}} | \nabla \widehat{v}_1^{(k)} |^2 \; \text{d}x 
	\] 
	and a discrete trace inequality, cf. \cite[Lemma 4.3]{EvansHughes:2013},
	\[ 
	\int_{\widehat{\Gamma}^{(k,\ell)}} | \nabla \widehat{v}_1^{(k)} |^2 \; \text{d}x 
	\lesssim
	\widehat{h}_{k\ell}^{-1} \int_{\widehat{\Omega}^{}} | \nabla \widehat{v}_1^{(k)} |^2 \; \text{d}x
	\] 
	and $\delta \lesssim 1$,
	we get
	\begin{equation}\label{eq:lem:D equivalence} 
		(\underline{v}_{1,e}^{(k)})^\top \widehat{D}_1^{(k)} \underline{v}_{1,e}^{(k)}
		\lesssim
		(1+\delta) \int_{\widehat{\Omega}^{}}^{}|\nabla \widehat{v}_1^{(k)}|^2 \; \text{d}x \lesssim
		\tilde{d}_1^{(k)}(\widehat{v}_{1,e}^{(k)}, \widehat{v}_{1,e}^{(k)}) = (\underline{v}_{1,e}^{(k)})^\top \widetilde{D}^{(k)}_1 \underline{v}_{1,e}^{(k)}, 
	\end{equation} 
	i.e., $\widehat{D}_1^{(k)}\lesssim \widetilde{D}^{(k)}_1$. 
	We obtain using the Cauchy-Schwarz inequality also 
	\[
	(\underline{v}_{1,e}^{(k)})^\top \widehat M^{(k)} \underline{v}_{1,e}^{(k)}
	=  \int_{\widehat \Omega} (v_{1}^{(k)})^2 \mathrm dx
	\lesssim (\underline{v}_{1,e}^{(k)})^\top \widetilde{D}_1^{(k)} \underline{v}_{1,e}^{(k)}.
	\] 
	Thus, we obtain $\widehat{D}_1^{(k)} + \widehat M^{(k)} \lesssim \widetilde{D}_1^{(k)}$.
	
	Now, we show the other direction. 
	By definition, 
	\[
	(\underline{v}_{1,e}^{(k)})^\top \widetilde{D}_1^{(k)} \underline{v}_{1,e}^{(k)} = \widetilde{d}_1^{(k)}(\widehat{v}_{1,e}^{(k)}, \widehat{v}_{1,e}^{(k)}) = 
	\int_{\widehat{\Omega}} |\nabla \widehat{v}_1^{(k)}|^2 \; \textrm{d}x 
	+ \int_{\widehat \Omega} (v_{1}^{(k)})^2 \mathrm dx.
	\] 
	Hence, we get
	\begin{align*}
		(\underline{v}_{1,e}^{(k)})^\top \widetilde{D}_1^{(k)} \underline{v}_{1,e}^{(k)}
		&\le \widehat{d}_e^{(k)}(\widehat{v}_{1,e}^{(k)}, \widehat{v}_{1,e}^{(k)})
		+ \int_{\widehat \Omega} (v_{1}^{(k)})^2 \mathrm dx
		\\& = (\underline{v}_{1,e}^{(k)})^\top (\widehat{D}_1^{(k)}+ \widehat M^{(k)}) \underline{v}_{1,e}^{(k)},
	\end{align*}
	which concludes the proof.
\end{proof}

Using~\cite[Lemma 4.13]{SchneckenleitnerTakacs:2019} and \eqref{eq:patchwise equivalent} 
we obtain
\begin{equation}
	\begin{aligned}\label{eq:pre combined}
		A^{(k)}+ \alpha H_k^{-2} M^{(k)} 
		&\eqsim \widehat{A}^{(k)}+ \alpha \widehat M^{(k)} 
		&\eqsim \widehat{D}^{(k)}+ \alpha \widehat M^{(k)} 
	\end{aligned}
\end{equation}
which furthermore characterizes the convergence of the chosen solver to compute the solution of~\eqref{eq:basis computation}. In fact, \eqref{eq:pre combined} implies together with \eqref{eq:phys para equivalent}, 
Lemma~\ref{lem:D equivalence}, and $\widetilde{D}_2^{(k)} = \widehat{D}_2^{(k)}$ the equivalence
\begin{equation}\label{eq:combined}
	\begin{aligned}
		P^{(k)} \eqsim (A_{\Delta \Delta}^{(k)})^{-1}.
	\end{aligned}
\end{equation}

Now, we show that the approximate Schur complement matrix
\[
\widehat{F} =
B_\Delta
P
B_\Delta^\top
+
B\Psi (\Psi^\top A \Psi)^{-1} \Psi^\top B^\top,
\]
is spectrally equivalent to
\[
F =
\begin{pmatrix}
	B_\Gamma & 0 
\end{pmatrix}
\begin{pmatrix}
	S & C_\Gamma^\top \\
	C_\Gamma & 
\end{pmatrix}^{-1}
\begin{pmatrix}
	B_\Gamma^\top \\ 0 
\end{pmatrix}
+
B\Psi (\Psi^\top A \Psi)^{-1} \Psi^\top B^\top,
\]
the Schur complement matrix from~\cite{SchneckenleitnerTakacs:2020}.
\begin{lemma}
	\label{lem:F equivalence}
	The equivalence $F\eqsim \widehat F$ holds.
\end{lemma}
\begin{proof}
	Observe that the splitting of the matrices $A^{(k)}$ into blocks
	and the definition of the Schur-complement matrix $S$ 
	immediately imply
	\[
	\begin{aligned}
		&
		\begin{pmatrix}
			B_\Gamma & 0 
		\end{pmatrix}
		\begin{pmatrix}
			S & C_\Gamma^\top \\
			C_\Gamma & 
		\end{pmatrix}^{-1}
		\begin{pmatrix}
			B_\Gamma^\top \\ 0 
		\end{pmatrix}
		=
		\begin{pmatrix}
			B & 0 
		\end{pmatrix}
		\begin{pmatrix}
			A & C^\top \\
			C & 
		\end{pmatrix}^{-1}
		\begin{pmatrix}
			B^\top \\ 0 
		\end{pmatrix}
		\\&\qquad= 
		\begin{pmatrix}
			B_\Delta & B_C & 0 
		\end{pmatrix}
		\begin{pmatrix}
			A_{\Delta \Delta} & A_{\Delta C} &  \\
			A_{C \Delta} & A_{CC} & I \\
			& I & 
		\end{pmatrix}^{-1}
		\begin{pmatrix}
			B_\Delta^\top \\ B_C^\top \\ 0  
		\end{pmatrix}
		\\&\qquad= 
		\begin{pmatrix}
			B_\Delta & B_C & 0 
		\end{pmatrix}
		\begin{pmatrix}
			A_{\Delta \Delta}^{-1} & 0 & -A_{\Delta \Delta}^{-1} A_{\Delta C} \\
			0 & 0 & I \\
			-A_{C\Delta} A_{\Delta \Delta}^{-1} & I & A_{\mathrm{C\Delta}}A_{\mathrm{\Delta C}}^{-1}A_{\mathrm{\Delta C}}-A_{\mathrm{CC}}
		\end{pmatrix}
		\begin{pmatrix}
			B_\Delta^\top \\ B_C^\top \\ 0  
		\end{pmatrix}
		= B_\Delta A_{\Delta \Delta}^{-1} B_\Delta^\top.
	\end{aligned}
	\]
	This shows
	\[
	F = 
	B_\Delta
	A_{\Delta\Delta}^{-1}
	B_\Delta^\top
	+
	B\Psi (\Psi^\top A \Psi)^{-1} \Psi^\top B^\top.
	\]
	Using~\eqref{eq:combined}, we obtain the desired result.
\end{proof}


In the following we show a similar equivalence for the scaled Dirichlet preconditioners.
\begin{lemma}
	\label{lem:Prec equivalence}
	The equivalence
	$ M_{\mathrm{sD}} := B_\Gamma \mathcal{D}^{-1}{S}\mathcal{D}^{-1}B_\Gamma^\top  \eqsim \widehat{M}_{\mathrm{sD}}$ holds.  
\end{lemma}
\begin{proof}
	Note that the Schur-complements
	$
	\widehat{S}_D^{(k)} = \widehat D_{\Gamma \Gamma}^{(k)} - \widehat D_{\Gamma \mathrm{I}}^{(k)} (\widehat{D}_{\mathrm{II}}^{(k)})^{-1} \widehat D_{\mathrm{I} \Gamma}^{(k)}
	$ 
	and
	$
	S^{(k)} = A_{\Gamma \Gamma}^{(k)} - A_{\Gamma \mathrm{I}}^{(k)} (A_{\mathrm{II}}^{(k)})^{-1} A_{\mathrm{I} \Gamma}^{(k)}
	$
	model the discrete harmonic extension, i.e., that extension that minimizes the energy norm. Since
	the norms on the physical domain and the parameter domain are equivalent, cf., e.g.,
	Lemma 4.13 in \cite{SchneckenleitnerTakacs:2019} and $D^{(k)} \eqsim A^{(k)}$, we know that
	\begin{align*}
		\| \underline v_\Gamma^{(k)} \|_{\widehat{S}^{(k)}}
		= \inf_{\underline w_{\mathrm I}^{(k)}}
		\left\| \begin{pmatrix} \underline w_{\mathrm I}^{(k)} \\ \underline v_\Gamma^{(k)} \end{pmatrix} \right\|_{\widehat D^{(k)}} 
		\eqsim \inf_{\underline w_{\mathrm I}^{(k)}}
		\left\| \begin{pmatrix} \underline w_{\mathrm I}^{(k)} \\ \underline v_\Gamma^{(k)} \end{pmatrix} \right\|_{A^{(k)}} 
		= 
		\| \underline v_\Gamma^{(k)} \|_{S^{(k)}}
	\end{align*}
	holds for all vectors $\underline v_\Gamma^{(k)}$ and for all $k=1,\ldots,K$.
	This shows $S \eqsim \widehat{S}_D$,
	which immediately implies
	$
	{M}_\mathrm{sD} = B_\Gamma \mathcal{D}^{-1}{S}\mathcal{D}^{-1}B_\Gamma^\top
	\eqsim
	B_\Gamma\mathcal{D}^{-1}\widehat{S}_D\mathcal{D}^{-1}B_\Gamma^\top = \widehat{M}_\mathrm{sD}.
	$
\end{proof}
Collecting all the results from the lemmas and theorems above we have the following condition number estimates for the system~\eqref{eq:star}, preconditioned with~$\mathcal{P}$, and the
system~\eqref{eq:basis computation}, preconditioned with~$(\widetilde{D}^{(k)})^{-1}$.
\begin{theorem}\label{thrm:fin}
	The following condition number estimate holds:
	\[
	\kappa(\mathcal{P} \mathcal{A}) \lesssim 
	\left(
	1 + \max_{k=1, \dots, K} \log \frac{H_k}{h_k}
	\right)^2,
	\]
	where $\kappa(Z):=\rho(Z)\rho(Z^{-1})$ is the condition number and
	$\rho(Z)$ is the spectral radius.
\end{theorem}
\begin{proof}
	From \cite[Theorem 4.1]{SchneckenleitnerTakacs:2020}, we know that 
	\[
	\sigma(M_\mathrm{sD} F) \subseteq
	\left[
	1, \left(1 + \max_{k=1, \dots, K} \log \frac{H_k}{h_k}\right)^2 \overline{\sigma}_1
	\right],
	\]
	with $\overline{\sigma}_1>0$, where $\sigma(Z)$ denotes the set of spectral values of the matrix $Z$. 
	(Since in the present paper, we do not discuss the dependence on
	$p$, we have dropped that parameter in the estimate.)
	Using Lemmas~\ref{lem:F equivalence} and \ref{lem:Prec equivalence}, we obtain
	\begin{equation}\label{eq:thrm:fin2}
		\sigma(\widehat M_\mathrm{sD} \widehat F) \subseteq 
		\left[
		\underline{\sigma}_2, \left(1 + \max_{k=1, \dots, K} \log \frac{H_k}{h_k}\right)^2 \overline{\sigma}_2
		\right],
	\end{equation}
	where $\underline{\sigma}_2, \overline{\sigma}_2>0$. The constants $\overline{\sigma}_1, \underline{\sigma}_2, \overline{\sigma}_2, \underline{\sigma}_3, \overline{\sigma}_3$ are independent of the Assumptions~\ref{ass:common element},~\ref{ass:non-singularity} and~\ref{ass:quasi-uniformity}.
	Equation~\eqref{eq:combined} yields $\sigma(P A_{\Delta \Delta})\subseteq [\underline{\sigma}_3, \overline{\sigma}_3]$, with $\underline{\sigma}_3, \overline{\sigma}_3>0$, which finishes the proof
	together with~\eqref{eq:thrm:fin2} in combination with Brezzi's theorem, cf.~\cite{Brezzi:1974}. 
\end{proof}

\begin{remark}
	\label{rem:no p-robustness}
	Note that in previous papers~\cite{SchneckenleitnerTakacs:2019,SchneckenleitnerTakacs:2020,SchneckenleitnerTakacs:2022}, the given bound for the condition number was explicit in the spline degree~$p$, precisely it was shown that the condition number grows not faster than $p(\log p)^2$. The numerical experiments have shown a linear growth in some cases, otherwise a sub-linear growth.
	That result cannot be transferred to the inexact solvers considered in this paper since the analysis presented in Lemma~\ref{lem:D equivalence} is not robust in $p$:
	In~\eqref{eq:lem:D equivalence}, we use that the penalty parameter $\delta$ is bounded
	uniformly, however it has to be chosen like $\delta\sim p^2$ in order to guarantee that the
	SIPG discretization is well-posed, cf.~\cite[Theorem~3.3]{Takacs:2019}.
\end{remark}

\section{Numerical results}
\label{sec:5}
In this section, we show numerical results to demonstrate the performance of the IETI-DP method proposed in Section~\ref{sec:4} for dG discretizations. We consider the following model problem: Find~$u \in H^1_0(\Omega)$ such that 
\begin{align*}
	-\Delta u &= 2\pi^2 \sin(\pi x) \sin(\pi y) &&\text{for}\quad (x,y) \in \Omega.
\end{align*}
The geometry shown in Figure~\ref{fig:computational domain} serves as a computational domain $\Omega$ on which we approximately solve our model problem. This domain is an approximated quarter annulus consisting of $8 \times 4 = 32$ patches, each of whom is parameterized by a B-spline mapping of degree 2. 
The set of patches is partitioned into three subsets, characterized by different colors as shown in the picture: green, red, and grey.

\begin{figure}[h]
	\centering
	\includegraphics[scale=0.3]{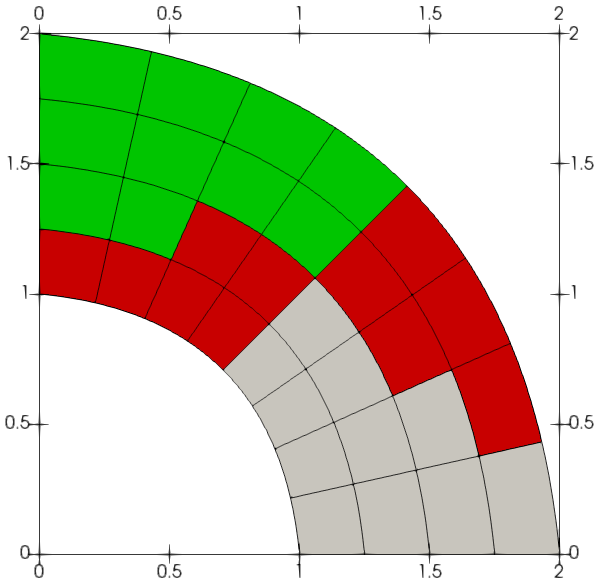}
	\caption{Computational domain and the decomposition into patches}
	\label{fig:computational domain}
\end{figure}
The numerical experiments are carried out with B-splines of maximum smoothness within each of the patches $\Omega^{(k)}$. 
The parameter $r$ is used to indicate the refinement level of the discretization space. 
The coarsest space, corresponding to $r = 0$, only consists of global polynomial functions on each patch. The subsequent discretizations $r=1,2,3,\dots$ are obtained as follows. First, we do $r$ dyadic refinements on all patches. Then, we uniformly refine the grey colored patches once more. The polynomial degrees for the discretization bases on the red patches are set to $p+1$ while on the other patches (green and grey) the degrees are set to $p$. 
In all cases, we consider splines of maximum smoothness within the patches, i.e.,
in $C^{p-1}$ and $C^p$, respectively. This is done in order to also obtain interfaces where neither
of the two adjacent trace spaces is a subspace of the other one.

\subsection{Setup and solving times of different solvers}
In this subsection, we report the setup and solving times of different IETI-DP formulations to get to the respective IETI-DP systems and to solve them. We compare the proposed method from this paper with a saddle point IETI-DP method using MINRES for the overall problem where we use sparse LU solvers for the local matrices 
\[
\mathcal{A}_\mathrm{C}^{(k)} := 
\begin{pmatrix}
	A^{(k)} & (C^{(k)})^{\top} \\
	C^{(k)}
\end{pmatrix}
\]
and sparse Cholesky factorizations for the matrices
$
A^{(k)}_\mathrm{II}
$
for all $k=1,\dots,K$, where $C^{(k)}$ denotes the matrix that selects the primal degrees of freedom of patch $\Omega^{(k)}$. The matrices $\mathcal{A}_\mathrm{C}^{(k)}$ and $A^{(k)}_\mathrm{II}$ are factorized in advance and used for the computation of the primal basis $\Psi$ and for preconditioning. The use of $\mathcal{A}_\mathrm{C}^{(k)}$ is the more classical approach in IETI-DP methods; we refer to~\cite{SchneckenleitnerTakacs:2020} for more details. We compute the LU and Cholesky factorizations using the Eigen library\footnote{\url{https://eigen.tuxfamily.org/}}. 
Furthermore, we compare the saddle point IETI-DP methods with the more standard IETI-DP method introduced and analyzed in~\cite{SchneckenleitnerTakacs:2020} that is based on a Schur complement formulation for which a PCG solver is utilized. 
Moreover, we consider the inexact IETI-DP method if we replace the preconditioner $P$ by $\mathcal{R}^{(\nu)}(P,A):=(I-(I-PA)^{\nu})A^{-1}$, $\nu \in \mathbb{N}$ which we realize as a Richardson scheme.
For a fair comparison between the methods we use only vertex values as primal degrees for all experiments. 
We indicate in the upcoming tables the proposed method from this paper with MFD, the inexact method with preconditioner $\mathcal{R}^{(2)}(P,A)$ by MFD-2, the saddle point IETI-DP method with local sparse solvers with MLU and the Schur complement based IETI-DP formulation with CGLU.  
All experiments have been carried out with the C++ library G+Smo~\cite{gismoweb}. The timings have been recorded on Radon1\footnote{\url{https://www.ricam.oeaw.ac.at/hpc/}}.


In the following, we show setup and application times of different quantities together with the overall solving time and outer iterations. We present the time required to compute the primal basis $\Psi$ and the times to set up the local preconditioners and the inverses in the scaled Dirichlet preconditioner. We use the symbol $\Psi$ in the tables to denote the computation time of $\Psi^{}$. Moreover, $\Upsilon_{S}$ and $\Theta_{S}$ denote the accumulated setup times of $P^{(k)}$ and $(D_{\mathrm{II}}^{(k)})^{-1}$ for MFD and MFD-2, and $(\mathcal{A}_\mathrm{C}^{(k)})^{-1}$, $(A_{\mathrm{II}}^{(k)})^{-1}$ for MLU and CGLU for all $k = 1,\dots, K$. Once the respective IETI-DP systems are set up we solve them. We report the accumulated application times of the respective preconditioners, the overall solving time and the outer iterations. We use in the forthcoming tables the abbreviations $\Upsilon_{A}$ and $\Theta_{A}$ to denote the accumulated application times of $P^{(k)}$ and $(D_{\mathrm{II}}^{(k)})^{-1}$ for MFD and MFD-2, and  $(\mathcal{A}_\mathrm{C}^{(k)})^{-1}$, $(A_{\mathrm{II}}^{(k)})^{-1}$ for MLU and CGLU, respectively for all $k = 1,\dots, K$. The overall solving times are denoted by solving and the outer iterations are abbreviated with ``it.''. Furthermore, we provide the total time required for the setup and solving the system by total.

\begin{table}[thp]
	\newcolumntype{L}[1]{>{\raggedleft\arraybackslash\hspace{-1em}}m{#1}}
	\centering
	\renewcommand{\arraystretch}{1.25}
	\begin{tabular}{r|c|L{3.em}|L{3.em}|L{3.em}|L{3.em}|L{3.em}|L{3.em}|L{3.em}|L{2.em}}
		\toprule
		\multicolumn{1}{c|}{\;\footnotesize $ $ }
		&\multicolumn{1}{c|}{\;\footnotesize $r$ \;}
		& \multicolumn{1}{c|}{\footnotesize $\Psi$}
		& \multicolumn{1}{c|}{\footnotesize $\Upsilon_S$}
		& \multicolumn{1}{c|}{\footnotesize $\Theta_{S}$}
		& \multicolumn{1}{c|}{\footnotesize $\Upsilon_A$} 
		& \multicolumn{1}{c|}{\footnotesize $\Theta_{A}$}
		& \multicolumn{1}{c|}{\footnotesize solving}
		& \multicolumn{1}{c|}{\footnotesize total}
		& \multicolumn{1}{c}{\footnotesize it.} \\
		\midrule
		\midrule
		MFD & $6$   
		& $18.1$ & $0.7$ & $0.1$ & $12.8$ & $5.9$ & $35.8 $ & $54.7$  & $505$\\
		MFD-2 & $ $   
		& $29.1$ & $1.1$ & $0.1$ & $24.5$ & $4.0$ & $39.5$ & $69.8$  & $372$\\
		MLU &  $ $ 
		& $2.0$ & $19.4$ & $2.6$ & $15.8$ & $4.9$ & $22.1$ & $46.1$ & $51$\\
		CGLU &  $ $ 
		& $2.0$ & $19.2$ & $2.6$ & $7.6$ & $2.1$ & $9.2$ & $ 33.3$ & $22$ \\
		\midrule
		MFD & $7$   
		& $88.9$ & $5.8$ & $0.6$ & $74.4$ & $45.3$ & $178.2$ & $273.5$ & $547$\\
		MFD-2 & $ $   
		& $134.1$ & $7.5$ & $0.6$ & $123.7$ & $31.9$ & $191.1$ & $333.3$ & $392$\\
		MLU &  $ $ 
		& $9.5$ & $111.3$ & $24.8$ & $85.0$ & $27.3$ & $117.3$ & $262.8$ & $57$\\
		CGLU &  $ $ 
		& $9.5$ & $110.8$ & $25.1$ & $39.6$ & $11.3$ & $48.3$ & $193.6$ & $24$ \\
		\midrule
		MFD & $8$   
		& $494.3$ & $80.4$ & $3.9$ & $460.1$ & $364.2$ & $1054.7$ & $1633.3$ & $578$\\
		MFD-2 & $ $   
		& $740.9$ & $90.6$ & $3.9$ & $744.2$ & $263.7$ & $1149.7$ & $1985.1$ & $417$\\
		MLU &  $ $ 
		& $46.2$ & $751.1$ & $285.0$ & $419.3$ & $137.4$ & $576.5$ & $1858.9$ & $59$ \\
		CGLU &  $ $ 
		& $46.1$ & $749.5$ & $284.5$ & $188.2$ & $54.8$ & $229.7$ & $1309.8$ & $24$ \\
		\bottomrule
	\end{tabular}
	\captionof{table}{Performance of the solvers; $p = 2$
		\label{tab:Times, p = 2}}
\end{table}

\begin{table}[thp]
	\newcolumntype{L}[1]{>{\raggedleft\arraybackslash\hspace{-1em}}m{#1}}
	\centering
	\renewcommand{\arraystretch}{1.25}
	\begin{tabular}{r|c|L{3.em}|L{3.0em}|L{3.em}|L{3.em}|L{3.em}|L{3.em}|L{3.em}|L{2.em}}
		\toprule
		\multicolumn{1}{c|}{\;\footnotesize $ $ }
		&\multicolumn{1}{c|}{\;\footnotesize $r$ \;}
		& \multicolumn{1}{c|}{\footnotesize $\Psi$}
		& \multicolumn{1}{c|}{\footnotesize $\Upsilon_{S}$}
		& \multicolumn{1}{c|}{\footnotesize $\Theta_{S}$}
		& \multicolumn{1}{c|}{\footnotesize $\Upsilon_{A}$} 
		& \multicolumn{1}{c|}{\footnotesize $\Theta_{A}$}
		& \multicolumn{1}{c|}{\footnotesize solving} 
		& \multicolumn{1}{c|}{\footnotesize total}
		& \multicolumn{1}{c}{\footnotesize it.}\\
		\midrule
		\midrule
		MFD & $6$   
		& $24.3$ & $0.7$ & $0.1$ & $ 15.0$ & $6.8$ & $50.2$ & $75.3$ & $573$ \\
		MFD-2 & $ $   
		& $40.5$ & $1.2$ & $0.1$ & $33.4$ & $5.0$ & $55.8$ & $97.6$ & $422$\\
		MLU &  $ $ 
		& $ 3.1$ & $45.6$ & $9.6$ & $26.5$ & $10.7$ & $39.3$ & $97.5$ & $55$ \\
		CGLU &  $ $ 
		& $3.0$ & $45.0$ & $9.6$ & $11.8$ & $4.2$ & $15.2$ & $72.9$ & $22$\\
		\midrule
		MFD & $7$   
		& $115.3$ & $4.2$ & $0.6$ & $81.5$ & $48.70$ & $227.1$ & $347.2$ & $605$\\
		MFD-2 & $ $   
		& $178.7$ & $6.4$ & $0.6$ & $163.5$ & $36.7$ & $260.1$ & $445.8$ & $452$\\
		MLU &  $ $
		& $20.0$ & $433.5$ & $135.5$ & $179.8$ & $61.41$ & $248.6$ & $837.5$ & $59$ \\
		CGLU &  $ $ 
		& $20.0$ & $432.0$ & $135.2$ & $80.9$ & $24.58$ & $99.8$ & $687.0$ & $24$ \\
		\midrule
		MFD & $8$   
		& $613.8$ & $50.9$ & $3.9$ & $529.2$ & $397.7$ & $1327.1$ & $1995.8$ & $659$\\
		MFD-2 & $ $   
		& $928.8$ & $62.3$ & $3.9$ & $937.6$ & $289.7$ & $1467.0$ & $2462.0$ & $478$\\
		MLU &  $ $ 
		& $114.3$ & $3246.3$ & $1578.9$ & $978.4$ & $329.4$ & $1337.3$ & $6276.9$ & $61$\\
		CGLU &  $ $ 
		& $114.3$ & $3237.4$ & $1570.3$ & $457.1$ & $137.7$ & $564.0$ & $5476.0$ & $26$ \\
		\bottomrule
	\end{tabular}
	\captionof{table}{Performance of the solvers; $p = 3$
		\label{tab:Times, p = 3}}
\end{table}

In the Tables~\ref{tab:Times, p = 2},~\ref{tab:Times, p = 3},~\ref{tab:Times, p = 4},~\ref{tab:Times, p = 5},~\ref{tab:Times, p = 6},~\ref{tab:Times, p = 7} and \ref{tab:Times, p = 8}, we present the setup and solving times for the mentioned IETI-DP formulations for the spline degrees $p = 2$ to $p = 8$, respectively. We see the advantage of sparse direct solvers for problems with multiple right-hand side vectors. The primal basis computation is quite faster compared to the use of an iterative solver. The setup of the local preconditioners $P^{(k)}$ requires less time compared to the factorization of $\mathcal{A}_{\mathrm{C}}^{(k)}$, in particular if we use a large spline degree $p$. Also the setup of $(D_{\mathrm{II}}^{(k)})^{-1}$ is by factors faster than to factorize the matrices $A_{\mathrm{II}}^{(k)}$. We observe that the application of the inexact solvers leads to significantly larger numbers of iterations, which in general leads to more time spent in the solving phase. Regarding the total time, MFD is the fastest method in almost all cases except in those with lower spline degree $p$ and coarser grids (where the exact local solves dominate). We observe from the tables that the application of $\mathcal{R}^{(2)}(P,A)$ is more expensive that the application of $P$. The smaller number of outer iterations for MFD-2 only reduces the computational costs slightly, compared to MFD.  
We also observe from the tables that the memory requirement for sparse direct solvers limits the problem size we are able to solve with the IETI-DP formulations MLU and CGLU. In fact, for spline degrees $p = 7$ and $p=8$ and refinement level $r=8$, MLU and CGLU run out of memory while MFD and MFD-2 deliver a solution for these larger problems. 

\begin{table}[thp]
	\newcolumntype{L}[1]{>{\raggedleft\arraybackslash\hspace{-1em}}m{#1}}
	\centering
	\renewcommand{\arraystretch}{1.25}
	\begin{tabular}{r|c|L{3.em}|L{3.em}|L{3.em}|L{3.em}|L{3.em}|L{3.em}|L{3.em}|L{2.em}}
		\toprule
		\multicolumn{1}{c|}{\;\footnotesize $ $ }
		&\multicolumn{1}{c|}{\;\footnotesize $r$ \;}
		& \multicolumn{1}{c|}{\footnotesize $\Psi$}
		& \multicolumn{1}{c|}{\footnotesize $\Upsilon_{S}$}
		& \multicolumn{1}{c|}{\footnotesize $\Theta_{S}$}
		& \multicolumn{1}{c|}{\footnotesize $\Upsilon_{A}$} 
		& \multicolumn{1}{c|}{\footnotesize $\Theta_{A}$}
		& \multicolumn{1}{c|}{\footnotesize solving}
		& \multicolumn{1}{c|}{\footnotesize total} 
		& \multicolumn{1}{c}{\footnotesize it.}\\
		\midrule
		\midrule
		MFD & $6$   
		& $34.6$ & $0.9$ & $0.2$ & $18.2$ & $8.4$ & $74.6$ & $110.1$ & $681$ \\
		MFD-2 & $ $   
		& $59.8$ & $1.50$ & $0.2$ & $48.3$ & $6.3$ & $83.8$ & $145.3$  & $508$\\
		MLU &  $ $ 
		& $4.3$ & $91.1$ & $15.4$ & $36.4$ & $14.3$ & $53.6$ & $164.4$ & $55$\\
		CGLU &  $ $ 
		& $4.3$ & $90.2$ & $15.3$ & $17.6$ & $6.2$ & $22.6$ & $132.4$ & $24$\\
		\midrule
		MFD & $7$   
		& $157.3$ & $5.3$ & $0.6$ & $102.5$ & $61.1$ & $339.0$ & $502.2$ & $729$\\
		MFD-2 & $ $   
		& $256.9$ & $8.02$ & $0.6$ & $46.1$ & $233.1$ & $385.2$ & $650.8$ & $556$\\
		MLU &  $ $ 
		& $27.7$ & $858.0$ & $188.2$ & $240.7$ & $78.4$ & $329.6$ & $1403.4$ & $59$ \\
		CGLU &  $ $ 
		& $27.7$ & $861.3$ & $188.2$ & $112.2$ & $32.7$ & $137.2$ & $1214.3$ & $25$\\
		\midrule
		MFD & $8$   
		& $796.6$ & $62.2$ & $3.9$ & $624.4$ & $472.9$ & $1785.2$ & $2647.8$ & $771$\\
		MFD-2 & $ $   
		& $1261.5$ & $76.7$ & $3.9$ & $1267.9$ & $351.2$ & $2036.1$ & $3378.2$ & $575$\\
		MLU &  $ $ 
		& $194.6$ & $10301.0$ & $1634.1$ & $1723.9$ & $395.2$ & $2165.6$ & $14295.3$ & $65$\\
		CGLU &  $ $ 
		& $194.54$ & $10284.0$ & $1625.01$ & $755.84$ & $155.1$ & $859.8$ & $12963.3$ & $26$\\
		\bottomrule
	\end{tabular}
	\captionof{table}{Performance of the solvers; $p = 4$
		\label{tab:Times, p = 4}}
\end{table}

\begin{table}[thp]
	\newcolumntype{L}[1]{>{\raggedleft\arraybackslash\hspace{-1em}}m{#1}}
	\centering
	\renewcommand{\arraystretch}{1.25}
	\begin{tabular}{r|c|L{3.em}|L{3.em}|L{3.em}|L{3.em}|L{3.em}|L{3.em}|L{3.em}|L{2.em}}
		\toprule
		\multicolumn{1}{c|}{\;\footnotesize $ $ }
		&\multicolumn{1}{c|}{\;\footnotesize $r$ \;}
		& \multicolumn{1}{c|}{\footnotesize $\Psi$}
		& \multicolumn{1}{c|}{\footnotesize $\Upsilon_{S}$}
		& \multicolumn{1}{c|}{\footnotesize $\Theta_{S}$}
		& \multicolumn{1}{c|}{\footnotesize $\Upsilon_{A}$} 
		& \multicolumn{1}{c|}{\footnotesize $\Theta_{A}$}
		& \multicolumn{1}{c|}{\footnotesize solving} 
		& \multicolumn{1}{c|}{\footnotesize total}
		& \multicolumn{1}{c}{\footnotesize it. } \\
		\midrule
		\midrule
		MFD & $6$   
		& $49.6$ & $1.0$ & $0.2$ & $23.6$ & $11.0$ & $115.0$ & $165.8$ & $835$\\
		MFD-2 & $ $   
		& $86.4$ & $1.8$ & $0.2$ & $71.6$ & $8.1$ & $127.4$ & $215.7$ & $628$\\
		MLU &  $ $ 
		& $6.1$ & $180.4$ & $30.5$ & $51.9$ & $20.2$ & $76.1$ & $293.1$ & $57$ \\
		CGLU &  $ $ 
		& $6.1$ & $178.6$ & $30.4$ & $25.4$ & $9.2$ & $32.9$ & $248.1$ & $24$ \\
		\midrule
		MFD & $7$   
		& $219.6$ & $6.2$ & $0.6$ & $125.0$ & $75.1$ & $492.1$ & $718.6$ & $888$\\
		MFD-2 & $ $   
		& $371.0$ & $9.7$ & $0.6$ & $327.9$ & $56.7$ & $557.0$ & $938.3$ & $668$\\
		MLU &  $ $ 
		& $35.1$ & $1389.7$ & $328.3$ & $307.8$ & $106.3$ & $428.8$ & $2181.9$ & $61$\\
		CGLU &  $ $ 
		& $35.1$ & $1383.9$ & $328.5$ & $143.7$ & $44.6$ & $178.8$ & $1926.4$ & $26$ \\
		\midrule
		MFD & $8$   
		& $1062.7$ & $74.5$ & $4.0$ & $784.0$ & $595.2$ & $2568.6$ & $3709.8$ & $967$\\
		MFD-2 & $ $   
		& $1787.9$ & $91.3$ & $4.0$ & $1725.5$ & $431.3$ & $2844.3$ & $4727.5$ & $699$\\
		MLU &  $ $ 
		& $263.2$ & $18354.5$ & $3562.5$ & $2372.9$ & $617.6$ & $3061.1$ & $25241.3$ & $65$\\
		CGLU &  $ $ 
		& $263.2$ & $18403.0$ & $3624.8$ & $1037.8$ & $236.2$ & $1205.7$ & $23496.7$ & $27$\\
		\bottomrule
	\end{tabular}
	\captionof{table}{Performance of the solvers; $p = 5$
		\label{tab:Times, p = 5}}
\end{table}

\begin{table}[thp]
	\newcolumntype{L}[1]{>{\raggedleft\arraybackslash\hspace{-1em}}m{#1}}
	\centering
	\renewcommand{\arraystretch}{1.25}
	\begin{tabular}{r|c|L{3.em}|L{3.em}|L{3.em}|L{3.em}|L{3.em}|L{3.em}|L{3.em}|L{2.em}}
		\toprule
		\multicolumn{1}{c|}{\;\footnotesize $ $ }
		&\multicolumn{1}{c|}{\;\footnotesize $r$ \;}
		& \multicolumn{1}{c|}{\footnotesize $\Psi$}
		& \multicolumn{1}{c|}{\footnotesize $\Upsilon_{S}$}
		& \multicolumn{1}{c|}{\footnotesize $\Theta_{S}$}
		& \multicolumn{1}{c|}{\footnotesize $\Upsilon_{A}$} 
		& \multicolumn{1}{c|}{\footnotesize $\Theta_{A}$}
		& \multicolumn{1}{c|}{\footnotesize solving} 
		& \multicolumn{1}{c|}{\footnotesize total}
		& \multicolumn{1}{c}{\footnotesize it. } \\
		\midrule
		\midrule
		MFD & $6$   
		& $ 70.6$ & $1.1$ & $0.2$ & $29.9$ & $13.6$ & $174.7$ & $256.6$ & $1067$\\
		MFD-2 & $ $   
		& $131.1$ & $2.1$ & $0.2$ & $110.0$ & $10.3$ & $199.0$ & $332.3$ & $800$\\
		MLU &  $ $ 
		& $7.6$ & $287.4$ & $57.8$ & $65.5$ & $27.5$ & $98.4$ & $451.3$ & $59$\\
		CGLU &  $ $ 
		& $7.6$ & $283.7$ & $57.5$ & $30.6$ & $11.5$ & $40.1$ & $389.0$ & $25$ \\
		\midrule
		MFD & $7$   
		& $310.3$ & $8.3$ & $0.65$ & $160.8$ & $95.4$ & $751.0$ & $1070.2$ & $1128$\\
		MFD-2 & $ $   
		& $554.7$ & $12.2$ & $0.7$ & $487.5$ & $71.1$ & $846.0$ & $1413.6$ & $845$\\
		MLU &  $ $ 
		& $49.2$ & $2654.4$ & $615.3$ & $445.5$ & $150.8$ & $616.8$ & $3935.6$ & $65$\\
		CGLU &  $ $ 
		& $49.2$ & $2644.3$ & $615.5$ & $195.2$ & $59.4$ & $241.6$ & $3550.5$ & $26$\\
		\midrule
		MFD & $8$   
		& $1458.1$ & $82.6$ & $3.9$ & $962.3$ & $721.5$ & $3598.5$ & $5143.1$ & $1172$\\
		MFD-2 & $ $   
		& $2544.3$ & $102.49$ & $3.9$ & $2480.5$ & $544.3$ & $4166.5$ & $6817.1$ & $886$\\
		MLU &  $ $ 
		& $317.8$ & $27693.7$ & $6599.9$ & $2896.1$ & $872.4$ & $3861.5$ & $38472.9$ & $69$\\
		CGLU &  $ $ 
		& $317.5$ & $27475.9$ & $6468.7$ & $1306.0$ & $371.4$ & $1537.0$ & $35799.1$ & $27$\\
		\bottomrule
	\end{tabular}
	\captionof{table}{Performance of the solvers; $p = 6$
		\label{tab:Times, p = 6}}
\end{table}

\begin{table}[thp]
	\newcolumntype{L}[1]{>{\raggedleft\arraybackslash\hspace{-1em}}m{#1}}
	\centering
	\renewcommand{\arraystretch}{1.25}
	\begin{tabular}{r|c|L{3.em}|L{3.em}|L{3.em}|L{3.em}|L{3.em}|L{3.em}|L{3.em}|L{2.em}}
		\toprule
		\multicolumn{1}{c|}{\;\footnotesize $ $ }
		&\multicolumn{1}{c|}{\;\footnotesize $r$ \;}
		& \multicolumn{1}{c|}{\footnotesize $\Psi$}
		& \multicolumn{1}{c|}{\footnotesize $\Upsilon_{S}$}
		& \multicolumn{1}{c|}{\footnotesize $\Theta_{S}$}
		& \multicolumn{1}{c|}{\footnotesize $\Upsilon_{A}$} 
		& \multicolumn{1}{c|}{\footnotesize $\Theta_{A}$}
		& \multicolumn{1}{c|}{\footnotesize solving}
		& \multicolumn{1}{c|}{\footnotesize total}
		& \multicolumn{1}{c}{\footnotesize it. } \\
		\midrule
		\midrule
		MFD & $6$   
		& $109.3$ & $1.1$ & $0.2$ & $41.9$ & $18.9$ & $288.5$ & $399.1$ & $1387$\\
		MFD-2 & $ $   
		& $208.2$ & $2.4$ & $0.2$ & $170.7$ & $13.9$ & $314.7$ & $525.5$ & $1043$\\
		MLU &  $ $ 
		& $8.89$ & $406.0$ & $82.3$ & $80.6$ & $35.7$ & $123.8$ & $620.9$ & $63$\\
		CGLU &  $ $ 
		& $8.9$ & $401.2$ & $81.7$ & $36.7$ & $14.5$ & $48.9$ & $540.6$ & $26$\\
		\midrule
		MFD & $7$   
		& $467.4$ & $5.9$ & $0.7$ & $204.4$ & $122.3$ & $1141.0$ & $1615.0$ & $1424$\\
		MFD-2 & $ $   
		& $858.4$ & $11.1$ & $0.7$ & $723.6$ & $92.2$ & $1279.0$ & $2149.2$ & $1068$\\
		MLU &  $ $ 
		& $62.2$ & $4579.7$ & $969.4$ & $574.1$ & $199.9$ & $800.9$ & $6412.3$ & $67$\\
		CGLU &  $ $ 
		& $62.2$ & $4565.6$ & $969.9$ & $260.5$ & $82.3$ & $326.6$ & $5924.2$ & $28$\\
		\midrule
		MFD & $8$   
		& $2149.1$ & $52.2$ & $4.2$ & $1254.1$ & $943.0$ & $5379.7$ & $7585.2$ & $1511$\\
		MFD-2 & $ $   
		& $ 3876.2$ & $75.7$ & $4.2$ & $3570.9$ & $700.7$ & $6126.4$ & $10082.5$ & $1123 $\\
		MLU &  $ $ 
		& \multicolumn{8}{c}{OoM} \\
		CGLU &  $ $ 
		& \multicolumn{8}{c}{OoM} \\
		\bottomrule
	\end{tabular}
	\captionof{table}{Performance of the solvers; $p = 7$
		\label{tab:Times, p = 7}}
\end{table}

\begin{table}[thp]
	\newcolumntype{L}[1]{>{\raggedleft\arraybackslash\hspace{-1em}}m{#1}}
	\centering
	\renewcommand{\arraystretch}{1.25}
	\begin{tabular}{r|c|L{3.em}|L{3.em}|L{3.em}|L{3.em}|L{3.em}|L{3.em}|L{3.em}|L{2.em}}
		\toprule
		\multicolumn{1}{c|}{\;\footnotesize $ $ }
		&\multicolumn{1}{c|}{\;\footnotesize $r$ \;}
		& \multicolumn{1}{c|}{\footnotesize $\Psi$}
		& \multicolumn{1}{c|}{\footnotesize $\Upsilon_{S}$}
		& \multicolumn{1}{c|}{\footnotesize $\Theta_{S}$}
		& \multicolumn{1}{c|}{\footnotesize $\Upsilon_{A}$} 
		& \multicolumn{1}{c|}{\footnotesize $\Theta_{A}$}
		& \multicolumn{1}{c|}{\footnotesize solving}
		& \multicolumn{1}{c|}{\footnotesize total}
		& \multicolumn{1}{c}{\footnotesize it. } \\
		\midrule
		\midrule
		MFD & $6$   
		& $254.7$ & $1.6$ & $0.2$ & $56.0$ & $26.2$ & $453.7$ & $710.2$ & $1831$ \\
		MFD-2 & $ $   
		& $494.6$ & $3.2$ & $0.2$ & $272.6$ & $19.6$ & $510.3$ & $1008.4$ & $1409$\\
		MLU &  $ $ 
		& $10.4$ & $563.9$ & $129.3$ & $108.7$ & $50.3$ & $169.5$ & $873.1$ & $73$\\
		CGLU &  $ $ 
		& $10.4$ & $557.7$ & $128.6$ & $48.3$ & $20.3$ & $66.1$ & $762.6$ & $30$\\
		\midrule
		MFD & $7$   
		& $1087.3$ & $7.9$ & $0.7$ & $277.8$ & $170.6$ & $1813.9$ & $2911.8$ & $1940$\\
		MFD-2 & $ $   
		& $2036.3$ & $14.1$ & $0.7$ & $1153.0$ & $129.7$ & $2080.6$ & $4131.7$ & $1471$\\
		MLU &  $ $ 
		& $78.9$ & $7550.6$ & $1449.7$ & $830.6$ & $282.8$ & $1153.1$ & $10232.3$ & $79$\\
		CGLU &  $ $ 
		& $78.9$ & $7530.6$ & $1449.8$ & $351.2$ & $109.6$ & $441.0$ & $9500.3$ & $31$\\
		\midrule
		MFD & $8$   
		& $5705.5$ & $62.0$ & $4.2$ & $1656.8$ & $1266.4$ & $6781.0$ & $12552.7$ & $2003$\\
		MFD-2 & $ $   
		& $9118.9$ & $89.4$ & $4.2$ & $5581.5$ & $976.5$ & $9749.3$ & $18961.7$ & $1548$\\
		MLU &  $ $ 
		& \multicolumn{8}{c}{OoM} \\
		CGLU &  $ $ 
		& \multicolumn{8}{c}{OoM} \\
		\bottomrule
	\end{tabular}
	\captionof{table}{Performance of the solvers; $p = 8$
		\label{tab:Times, p = 8}}
\end{table}

\subsection{Relative condition numbers of the local problems}

In Table~\ref{tab:localCondNumbers}, we present the relative condition numbers of the local problems, as estimated by the PCG algorithm. We investigate the condition numbers of $P^{(12)}A_{\Delta\Delta}^{(12)}$ since this is one of the generic patches that do not touch the Dirichlet boundary. Patch $\Omega^{(12)}$ is the gray patch that shares a vertex with a green patch. We observe that the condition number and number of iterations are basically robust in the grid size. However, we see a significant growth of the condition number with respect to the spline degree $p$, particularly for larger values of $p$. For $p\ge4$, the growth seems to be cubic in $p$. These findings seem to confirm the statement from Remark~\ref{rem:no p-robustness} that a $p$-robust analysis of the proposed approach is not possible.

\begin{table}[thp]
	\newcolumntype{L}[1]{>{\raggedleft\arraybackslash\hspace{-1em}}m{#1}}
	\centering
	\renewcommand{\arraystretch}{1.25}
	\begin{tabular}{c|L{3.em}|L{2.em}|L{3.em}|L{2.em}|L{3.em}|L{2.em}|L{3.em}|L{2.em}|L{3.em}|L{2.em}}
		\toprule
		\multicolumn{1}{c|}{\;\footnotesize $ $ }
		& \multicolumn{2}{c|}{\;\footnotesize $p=2$ }
		& \multicolumn{2}{c|}{\;\footnotesize $p=3$ }
		& \multicolumn{2}{c|}{\;\footnotesize $p=4$ }
		& \multicolumn{2}{c|}{\;\footnotesize $p=5$ }
		& \multicolumn{2}{c}{\;\footnotesize $p=6$ } \\
		\midrule
		\multicolumn{1}{c|}{\;\footnotesize $r$ \;}
		& \multicolumn{1}{c|}{\footnotesize $\kappa$}
		& \multicolumn{1}{c|}{\footnotesize it.}
		& \multicolumn{1}{c|}{\footnotesize $\kappa$}
		& \multicolumn{1}{c|}{\footnotesize it.} 
		& \multicolumn{1}{c|}{\footnotesize $\kappa$}
		& \multicolumn{1}{c|}{\footnotesize it.} 
		& \multicolumn{1}{c|}{\footnotesize $\kappa$}
		& \multicolumn{1}{c|}{\footnotesize it.} 
		& \multicolumn{1}{c|}{\footnotesize $\kappa$}
		& \multicolumn{1}{c}{\footnotesize it.}\\
		\midrule
		\midrule
		$6$ & 92.6 & 61
		& 98.8 & 63 & 122.8 & 70 & 222.8 & 82 & 438.7 & 98 \\
		$7$ & 92.3 & 61
		& 102.0 & 63 & 125.2 & 70 & 225.4 & 83 & 427.5 & 97 \\
		$8$ & 90.8 & 61
		& 101.9 & 63 & 123.9 & 70 & 230.1 & 83 & 422.0 & 98 \\
		\bottomrule
	\end{tabular}
	\captionof{table}{Condition numbers of $P^{(12)}A_{\Delta\Delta}^{(12)}$ estimated by PCG and number of PCG iterations to solve the corresponding problem to get a primal basis in the MFD method 
		\label{tab:localCondNumbers}}
\end{table}

\section{Conclusions}
\label{sec:6}
We have developed a IETI-DP method for dG discretizations that allows to incorporate the Fast Diagonalization method as inexact local solver. The local spaces in the IETI-DP algorithm for dG discretizations do not have tensor product structure that is required for FD. To circumvent this problem, we split a local function space into two function spaces that are orthogonal with respect to the localized dG-norm. One space has tensor product structure and the other one is a one-dimensional edge space. We have shown that the convergence bound with respect to $h$ we have developed in~\cite{SchneckenleitnerTakacs:2020} is maintained.
The presented method requires in almost all cases less time in total (setup and solving) and allows to reduce the memory footprint compared to the use of sparse direct solvers for the local problems. The numerical experiments indicate that MINRES is the preferable option to solve the resulting linear system. To obtain a fully iterative method, one has to find a suitable preconditioner for the coarse problem, which is left for future work.

Another future research direction is the extension of the proposed method to 3D problems. This might be very promising because the natural 3D extension of the local preconditioner $P^{(k)}$, and in particular its fundamental blocks $\widetilde{D}_1^{(k)}$ and $\widetilde{D}_2^{(k)}$, could be inverted very efficiently. Indeed, as discussed in \cite{SangalliTani:2016}, the Fast Diagonalization method used to invert $\widetilde{D}_1^{(k)}$ is even more appealing in 3D than in 2D. As for $\widetilde{D}_2^{(k)}$, in the present paper it is essentially a mass matrix on the patch edges, and can be effectively handled with a direct solver. In 3D, the corresponding mass matrix on the faces of the patch could be still inexpensively inverted thanks to its Kronecker structure. The non-trivial issue of a 3D extension is related to the choice of primal degrees of freedom in the IETI-DP framework. Indeed, it is known that simply taking the corners of each patch may lead to a worsening of the scalability. On the other hand, with the introduction of more suited primal degrees of freedom, e.g. face averages, it is not obvious how to maintain the tensor structure of the local problems, which is a key ingredient of the proposed approach.

\mbox{}\\[-2em]

\textbf{Acknowledgments.}
The third author was supported by the Austrian Science Fund (FWF): S117 and 
W1214-04. The fourth author has received support from the Austrian Science Fund (FWF): P31048. The first and second authors were partially supported by the Italian Ministry of Education, University and Research (MIUR) through the  ``Dipartimenti di Eccellenza Program (2018-2022) -
Dept. of Mathematics, University of Pavia''. These supports are
gratefully acknowledged. The first, second and fifth authors
are members of the Gruppo Nazionale Calcolo Scientifico-Istituto
Nazionale di Alta Matematica (GNCS-INDAM).

%
%
\bibliography{literature}

\begin{thebibliography}{10}

\bibitem{Arnold:1982}
D.~Arnold.
\newblock An interior penalty finite element method with discontinuous
  elements.
\newblock {\em SIAM J. Numer. Anal.}, 19(4):742 -- 760, 1982.

\bibitem{BadiaMartin:2015}
S.~Badia, A.~F. Mart\'in, and J.~Principe.
\newblock On the scalability of inexact balancing domain decomposition by
  constraints with overlapped coarse/fine corrections.
\newblock {\em Parallel Comput.}, 50:1 -- 24, 2015.

\bibitem{BozyMontardiniSangalliTani:2020}
M.~Bosy, M.~Montardini, G.~Sangalli, and M.~Tani.
\newblock A domain decomposition method for isogeometric multi-patch problems
  with inexact local solvers.
\newblock {\em Comput. Math. with Appl.}, 80(11):2604 -- 2621, 2020.

\bibitem{Brezzi:1974}
F.~Brezzi.
\newblock {On the existence, uniqueness and approximation of saddle-point
  problems arising from lagrangian multipliers}.
\newblock {\em ESAIM: Math. Model. Numer. Anal.}, 8(R2):129 -- 151, 1974.

\bibitem{Boor}
C.~de~Boor.
\newblock {\em A Practical Guide to Splines}, volume~27.
\newblock Applied Mathematical Sciences, New York: Springer, 1978.

\bibitem{Dohrmann:2007}
C.~R. Dohrmann.
\newblock {An approximate BDDC preconditioner}.
\newblock {\em Numer. Linear Algebra Appl.}, 14(2):149 -- 168, 2007.

\bibitem{DryjaGalvis:2013}
M.~Dryja, J.~Galvis, and M.~Sarkis.
\newblock A {FETI-DP} preconditioner for a composite finite element and
  discontinuous {G}alerkin method.
\newblock {\em SIAM J. Numer. Anal.}, 51(1):400 -- 422, 2013.

\bibitem{EvansHughes:2013}
J.~Evans and T.~J.~R. Hughes.
\newblock Explicit trace inequalities for isogeometric analysis and parametric
  hexahedral finite elements.
\newblock {\em Numer. Math.}, 123(2):259 -- 290, 2013.

\bibitem{FarhatRoux:1991a}
C.~Farhat and F.-X. Roux.
\newblock A method of finite element tearing and interconnecting and its
  parallel solution algorithm.
\newblock {\em Int. J. Numer. Methods Eng.}, 32(6):1205 -- 1227, 1991.

\bibitem{GolubLoan:2013}
G.~H. Golub and C.~F. Van~Loan.
\newblock {\em {Matrix Computations}}.
\newblock The John Hopkins University Press, 2013.

\bibitem{HoferLanger:2017c}
C.~Hofer and U.~Langer.
\newblock Dual-primal isogeometric tearing and interconnecting solvers for
  multipatch {dG-IgA} equations.
\newblock {\em Comput. Methods Appl. Mech. Eng.}, 316:2 -- 21, 2017.

\bibitem{Hofer:2019a}
C.~Hofer and U.~Langer.
\newblock {\em Dual-Primal Isogeometric Tearing and Interconnecting Methods},
  pages 273 -- 296.
\newblock Springer International Publishing, Cham, 2019.

\bibitem{HoferLanger:2019b}
C.~Hofer, U.~Langer, and I.~Toulopoulos.
\newblock Isogeometric analysis on non-matching segmentation: discontinuous
  {G}alerkin techniques and efficient solvers.
\newblock {\em J. Appl. Math. Comput.}, 61(1):297 -- 336, 2019.

\bibitem{CottrellHughes:2005}
T.~J.~R. Hughes, J.~Cottrell, and Y.~Bazilevs.
\newblock Isogeometric analysis: {CAD}, finite elements, {NURBS}, exact
  geometry and mesh refinement.
\newblock {\em Comput. Methods Appl. Mech. Eng.}, 194(39):4135 -- 4195, 2005.

\bibitem{KlawonnRheinbach:2007}
A.~Klawonn and O.~Rheinbach.
\newblock Inexact {FETI-DP} methods.
\newblock {\em Int. J. Numer. Methods Eng.}, 69(2):284 -- 307, 2007.

\bibitem{Kleiss:2012}
S.~K. Kleiss, C.~Pechstein, B.~J{\"u}ttler, and S.~Tomar.
\newblock {IETI}-isogeometric tearing and interconnecting.
\newblock {\em Comput. Methods Appl. Mech. Eng.}, 247-248(11):201 -- 215, 2012.

\bibitem{LiWidlund:2007}
J.~Li and O.~B. Widlund.
\newblock {On the use of inexact subdomain solvers for BDDC algorithms}.
\newblock {\em Comput. Methods Appl. Mech. Eng.}, 196(8):1415 -- 1428, 2007.

\bibitem{MandelDohrmannTezaur:2005a}
J.~Mandel, C.~R. Dohrmann, and R.~Tezaur.
\newblock An algebraic theory for primal and dual substructuring methods by
  constraints.
\newblock {\em Appl. Numer. Math.}, 54(2):167 -- 193, 2005.

\bibitem{gismoweb}
A.~Mantzaflaris, R.~Schneckenleitner, S.~Takacs, and others~(see website).
\newblock {G+Smo (Geometry plus Simulation modules)}.
\newblock \url{http://github.com/gismo}, 2022.

\bibitem{Pechstein:2013a}
C.~Pechstein.
\newblock {\em Finite and Boundary Element Tearing and Interconnecting Solvers
  for Multiscale Problems}.
\newblock Springer, Heidelberg, 2013.

\bibitem{SandeManni:2019}
E.~Sande, C.~Manni, and H.~Speleers.
\newblock Sharp error estimates for spline approximation: Explicit constants,
  $n$-widths, and eigenfunction convergence.
\newblock {\em Math. Models Methods Appl. Sci.}, 29(06):1175 -- 1205, 2019.

\bibitem{SangalliTani:2016}
G.~Sangalli and M.~Tani.
\newblock Isogeometric preconditioners based on fast solvers for the
  {S}ylvester equation.
\newblock {\em SIAM J. Sci. Comput.}, 38(6):A3644 -- A3671, 2016.

\bibitem{SchneckenleitnerTakacs:2019}
R.~Schneckenleitner and S.~Takacs.
\newblock {Condition number bounds for IETI-DP methods that are explicit in $h$
  and $p$}.
\newblock {\em Math. Models Methods Appl. Sci.}, 11(30):2067 -- 2103, 2020.

\bibitem{SchneckenleitnerTakacs:2020}
R.~Schneckenleitner and S.~Takacs.
\newblock Convergence theory for {IETI-DP} solvers for discontinuous {G}alerkin
  {I}sogeometric {A}nalysis that is explicit in $h$ and $p$.
\newblock {\em Comput. Methods Appl. Math.}, 22(1):199 -- 225, 2022.

\bibitem{SchneckenleitnerTakacs:2022}
R.~Schneckenleitner and S.~Takacs.
\newblock {IETI-DP methods for discontinuous Galerkin multi-patch Isogeometric
  Analysis with T-junctions}.
\newblock {\em Comput. Methods Appl. Mech. Eng.}, 393(Article 114694), 2022.

\bibitem{Takacs:2019}
S.~Takacs.
\newblock {Discretization error estimates for discontinuous Galerkin
  Isogeometric Analysis}.
\newblock {\em Appl. Anal.}, 2021.
\newblock Online first.

\end{thebibliography}

\end{document}